\documentclass[leqno]{amsart}
\usepackage{amsmath}        
\usepackage{amsfonts}
\usepackage{enumerate}
\usepackage{amsthm}
\usepackage{amssymb}
\usepackage{amscd}
\usepackage[all]{xy}
\usepackage[titletoc,toc,title]{appendix}

\newtheorem{thm}{Theorem}[section] 

\newtheorem{prop}[thm]{Proposition}
\newtheorem{cor}[thm]{Corollary}
\newtheorem{lemma}[thm]{Lemma}

\theoremstyle{definition}

\theoremstyle{definition}

\newtheorem*{rem}{Remark}
\newtheorem*{rems}{Remarks}
\newtheorem{Rem}[thm]{Remark}
\newtheorem{Rems}[thm]{Remarks}
\newtheorem{example}[thm]{Example}
\newtheorem{examples}[thm]{Examples}

\def\sk1{\vskip 10pt}
\def\newline{\hfil\break}

\def\ds{\displaystyle}

\def\x {\times}
\def\subeq{\subseteq}
\def\supeq{\supseteq}
\def\Sigmacirc {{^{\circ} \Sigma ^{n}}}

\def\thra{\twoheadrightarrow}

\def\Z{\mathbb Z}

\def\R{\mathbb R}

\def\del{\partial}

\def\sk{\vskip}
\def\bF{\bf F}

\numberwithin{equation}{section}

\title{Higher horospherical limit sets for $G$-modules over $CAT(0)$-spaces}

\author{Robert Bieri and Ross Geoghegan}

\address{\noindent Robert Bieri, Fachbereich Mathematik,
Johann Wolfgang Goethe-Universit\"at Frankfurt,
D-60054 Frankfurt am Main,
Germany \newline
\vskip 1pt
and
\newline
\vskip 1pt
Department of Mathematical Sciences,
Binghamton University (SUNY),
Binghamton, NY 13902-6000, USA
\newline
\vskip 3pt
Ross Geoghegan, Department of Mathematical Sciences,
Binghamton University (SUNY),
Binghamton, NY 13902-6000, USA}
\email{bieri@math.uni-frankfurt.de,
ross@math.binghamton.edu}

\subjclass[2010]{Primary 20F65; Secondary 20J05, 59D19}

\date{October 15, 2018}

\keywords{Bieri-Neumann-Strebel invariant, $CAT(0)$ space, horospherical
limit point, Novikov homology}

\begin{document}
\fontsize{12}{13pt} \selectfont  

\maketitle


\maketitle

\begin{abstract}
The $\Sigma$-invariants of Bieri-Neumann-Strebel and Bieri-Renz involve an
action of a discrete group $G$ on a geometrically suitable space $M$. In
the early versions, $M$ was always a finite-dimensional Euclidean space
on which $G$ acted by translations. A substantial literature exists on
this, connecting the invariants to group theory and to tropical geometry
(which, actually, $\Sigma $-theory anticipated). More recently, we have
generalized these invariants to the case where $M$ is a proper $CAT(0)$
space on which $G$ acts by isometries. The ``zeroth stage'' of this
was developed in our paper \cite{BGe16}. The present paper provides
a higher-dimensional extension of the theory to the ``$n$th stage''
for any $n$.
\end{abstract} 

\section{Introduction}\label{S:1}

\subsection{Background}\label{subsection0} 

The Bieri-Neumann-Strebel invariant $\Sigma (G,{\Z})$ of a finitely
generated group $G$ is a certain subset of the sphere at infinity of ${\R}^d$, where $d$ is the rank of the abelianization of $G$. Over the years this elusive subset has been computed in many cases, and has been related to a variety of issues in group theory and tropical geometry. The first major generalization was the  higher-dimensional Bieri-Renz invariant $\Sigma ^{n}(G;A)$, where $A$ is a $G$-module of type $FP_n$. (The case $n=1$ with $A$ the trivial $G$-module ${\Z}$ gives back the original $\Sigma (G,{\Z})$).

The invariant $\Sigma ^{n}(G;A)$ is intrinsically associated with
the natural action of $G$ on ${\R}^d$ by translations. This led us to
a generalization of the fundamental idea, in which, given $G$ and a
$G$-module $A$ of type $FP_n$, the translation action of $G$ on ${\R}^d$
is replaced by an isometric action of $G$  on an arbitrary proper $CAT(0)$
space $M$, leading us to a subset of the boundary $\del M$ playing the
role previously played by the sphere at infinity. By clear analogy,
we call this subset $\Sigma ^{n}(M;A)$.

Even the case $\Sigma ^{0}(M;A)$ has turned out to be remarkably
interesting, Although computation is still in its infancy, it has already
been deeply linked to buildings associated with certain arithmetic groups,
where $M$ in that case is a symmetric space. The basic theory of $\Sigma
^{0}(G;A)$ is set  out in our paper \cite{BGe16}.

In the present paper we set out the corresponding theory of $\Sigma
^{n}(G;A)$, thus exhibiting the natural place of $\Sigma ^{0}(M;A)$ within
a richer theory. We prove appropriate analogs of theorems already known
in the ``classical'' (i.e. Euclidean) case.  This is far from routine,
and new methods have to be developed. Besides the basic theory, we find
a product formula for $\Sigma $ invariants, and an interpretation of
the whole theory in terms of Novikov homology.


\subsection{Quick definition of $\Sigma ^{n}(G;A)$}

We prefer to give this definition in context later in the paper, but
for readers who wish to know now, we include a short version here.

The given $G$-module $A$ has type $FP_n$, so there is a free resolution
${\bf F}\twoheadrightarrow A$ with chosen finitely generated $n$-skeleton. We
choose a base point $b\in M$, and we define a $G$-equivariant ``control
map'' $h:{\bf F}\to fM$ where $fM$ denotes the $G$-set of finite subsets
of $M$.

The definition of this map $h$ proceeds in stages: 
\begin{enumerate}[1.]
\item First we define $h:{\Z}G\to fM$ taking $\lambda \in {\Z}G$ to the finite subset $h(\lambda ):=\text{supp}(\lambda )\cdot b \subseteq M$.
\item We extend this to a finitely generated free $G$-module $F$ with chosen basis by defining $h:F\to fM$ separately as above on each ${\Z}G$-summand and then taking the union.
\item This is then applied to the $n$-skeleton of the resolution $\bf F$.
\end{enumerate}

Roughly,  $\Sigma ^{n}(M;A)$ consists of those points $e\in \del M$ whose
horoballs at $e$, $H_{e}\subseteq \del M$, have the property that every
cycle $z$ in the $(n-1)$-skeleton ${\bf F}^{(n-1)}$ with $h(z)\subseteq H_e$
bounds a chain in ${\bf F}^{(n)}$ with $h(c)\subseteq H'_e$ where $H'_e$
depends only on (and is only slightly larger than) $H_e$.  

\subsection{The dynamical invariant}

Given a base point $b\in M$ we use the Busemann function $\beta _{e}:M\to {\R}$,
normalized by $\beta _{e}(b)=0$, to measure the effect of the $G$-action
and chain endomorphisms on the images $h(c)$ of the elements
$c\in {\bf F}$ of dimension $\leq n$. Specifically, we consider chain
endomorphisms $\varphi :{\bf F}\to {\bf F}$ with the properties

\begin{enumerate}[1] \item $\varphi$ lifts the identity map of $A$;
\item The difference $\beta _{e}(h(\varphi (c)))-\beta _{e}(h(c))$ has a
positive lower bound as $c$ runs through the $n$-skeleton ${\bf F}^{(n)}$
of ${\bf F}$.  
\end{enumerate}

\noindent We call $\varphi $ a ``push'' of the $n$-skeleton towards
$e$. In the classical case, Theorem 4.1 of \cite{BRe88} shows
that the existence of such a $G$-equivariant push is equivalent
to $e\in \Sigma ^{n}(M;A)$ -- a key fact often referred to as the
``$\Sigma ^{n}$ criterion". Here in the $CAT(0)$ situation we find
that the set of boundary points $e\in \Sigma ^{n}(M;A)$ for which a
$G$-equivariant push $\varphi :{\bf F}^{(n)}\to {\bf F}^{(n)}$ towards $e$
exists is in general a proper subset -- potentially interesting but not
sufficiently closely related to $\Sigma ^{n}(M;A)$ since it vanishes in
some of the most interesting examples. In \cite{BGe16} (which was about
$\Sigma ^{0}(M;A)$ in the $CAT(0)$ case) we introduced the notion of
$G$-finitary homomorphisms between $G$-modules. These are more general
than $G$-homomorphisms but still share their coarse metric properties.
We define the subset $\Sigmacirc (M;A)\subseteq \Sigma ^{n}(M;A)$
to be the set of those points $e\in \del M$ such that there is a
$G$-finitary push of the $n$-skeleton towards $e$, and we call it the
``dynamical invariant". By using $G$-finitary chain homotopies in higher
dimensional homological algebra arguments, we prove:

\begin{thm} ($\Sigma ^{n}$ Criterion) \label{criterion}
\begin{equation*} \Sigmacirc(M;A) = \{e\in \partial M\mid \text{\rm cl}(Ge)\subeq 
\Sigma ^{n}(M;A)\}
\end{equation*}
\end{thm}

(Here, closure is taken in the cone topology on $\del M$.)

In the classical theory $G$ acts trivially on $\del M$, hence Theorem
\ref{criterion} is a true generalization of the classical $\Sigma ^{n}$
criterion, and, just as back then, it is again the fundamental tool for
all further results.

\subsection{The main results}\label{subsection1}

\begin{thm} \label{Theorem C}\label{T:7.9} Assume that the action $\rho$
is cocompact and that its orbits are discrete subsets of $M$. Then
$\Sigma^n(M;A)=\partial M$ if and only if $A$ has type $FP_n$ as a
$G_b$-module, where $b$ is any point of $M$ and $G_b$ denotes its
stabilizer.
\end{thm}

We have openness theorems as follows:
\begin{thm}\label{Theorem D}
\begin{enumerate}[(i)]
\item  With $\rho \in \text{\rm Hom}(G, \text{\rm Isom}(M))$ an isometric action of $G$ on $M$ as above, if
$\Sigma ^{n}(_{\rho }M;A)=\partial M$ then there is a neighborhood $N$ of
$\rho$ in this space (with the compact-open
topology) such that $\Sigma ^{n}(_{\rho '}M;A)=\partial M$ for all $\rho
'\in N$.
\item $\Sigmacirc (_{\rho }M;A)$ is open in the Tits metric topology on $\del M$.
\end{enumerate}
\end{thm}

\noindent $\Sigma ^{n}_{\rho }(M;A)$ is not in general open in $\del M$ with the cone topology.

The next two theorems concerning the dynamical invariant $\Sigmacirc (M;A)$
-- a homological description and a product formula -- are new, even in the 
$0$-dimensional case \cite{BGe16}.

Given a base point $b\in M$ and a boundary point $e\in \del M$ we
write $\widehat {{\Z}G}^e$ for the set all infinite sums $\sigma =
\Sigma _{g\in G}a_{g}g$ with the property that each horoball $HB_e$
at $e$ contains all but a finite subset of $\text{supp}(\sigma )b\subseteq
M$. This $\widehat {{\Z}G}^e$ is a right $G$-module which we call the
{\it Novikov module at} $e$.

\begin{thm}\label{novikov1} If $A$ is a ${\Z}G$-module of type $FP_n$ then\footnote{The special case of Theorem \ref{novikov1} when $M=G_{ab}\otimes {\R}$
is Euclidean is proved in Pascal Schweitzer's appendix to \cite{B07}.}
$$\Sigmacirc (M;A) =\{e\in \del M\mid \text{ \rm Tor}_{k}(\widehat {{\Z}G}^{e'},A)=0 \text { for all } e'\in \text{\rm cl}Ge \text{ and all } k\leq n.\}$$
\end{thm}

If $(M',A')$ is second pair consisting of a proper $CAT(0)$ space $M'$
and an abelian group $A'$, both acted on by a group $H$, then we have
a corresponding pair $(M\times M', A\otimes A')$ with the obvious
$G\times H$ action. Assuming $A$ and $A'$ are of type $FP_n$ as $G$-
[resp. $H$]-modules, we can take advantage of the identification $\del
(M\times M')=\del  M \star \del M'$ to ask for a formula expressing
$\Sigma ^{*}(M\times M ';A\otimes A')$ in terms of  $\Sigma ^{*}(M;A)$
and  $\Sigma ^{*}(M';A')$.  While there is no intrinsic relationship
between the subsets $\Sigma ^{*}(M;A)$ and $\Sigma ^{*}(M';A')$ of $\del
(M\times M')$, the tensor product $A\otimes A'$ with ground ring $\Z$
could be zero. Hence there is no hope for a simple formula without
restrictions on the modules.  For this reason we replace the ground ring
$\Z$ by a field $K$ in our product formula, and we interpret the $\Sigma
$-invariants correspondingly. As usual, formulas for these invariants
are best expressed in terms of their complements $\Sigma ^{c}=\del M
- \Sigma$.

\begin{thm}\label{productthm} Let $K$ be a field and let $A$, $A'$ be $KG$- (resp. $KH$-)modules of type $F_n$, with the additional assumption that ${^ \circ}\Sigma ^{0}(M;A)=\del M$
and $^{\circ}\Sigma ^{0}(M';A')=\del M'$. Then
$$ \Sigmacirc(M\times M';A\otimes _{K} A')^c = \bigcup^n_{p=0} 
{^{\circ} \Sigma ^{p}(M;A)^c * 
{^{\circ} \Sigma ^{n-p}}}(M';A')^c.$$
\end{thm}

\begin{Rems}
\newline
\begin{enumerate}[(1)]
\item Theorem \ref{productthm} extends our product formula for 
$\Sigma ^{n}(G\times H;K)$ in \cite{BGe10}.
\vskip 5pt
\item In the discrete case the assumption that $\Sigma ^{0}(M;A)=\del M$
is equivalent to saying that the $G$-action on $M$ is cocompact and that
$A$ is finitely generated over any point stabilizer; see \cite{BGe16}. For
more details see the remarks in Section \ref{S:10}.
\vskip 5pt
\item An early (sometimes forgotten) product formula for the original
Bieri-Strebel invariant, defined for modules over finitely generated
abelian groups, is not covered by Theorem \ref{productthm}. It asserts
that for such groups $G$ and $H$ and arbitrary $KG$- and
$KH$-modules $A$ and $A'$,
\begin{equation*} \Sigma ^{0}(G\times H; A\otimes _{K} A')^c = \Sigma ^{0}(G;A)^c *\Sigma ^{0}(H;A')^c \tag{*}
\end{equation*}

\noindent It is a fact that when $G$ is finitely generated and abelian then $\Sigma ^{n}(G;A)= \Sigma ^{0}(G;A)$ for all $n\geq 0$. So, in this ``classical case" Theorem \ref{productthm} holds without any restriction on $\Sigma ^{0}$.    
\vskip 5pt
The formula $(*)$ appeared in \cite{BGr82} as the key to
proving that every metabelian group of type $FP_{\infty}$ is virtually
of type $FP$. It would be highly interesting to have a general formula
for $ \Sigma ^{n}(G\times H; A\otimes _{K} A')$ that explains the role
of $\Sigma ^{0}(G;A)$ and $\Sigma ^{0}(H;A')$.

\end{enumerate}
\end{Rems}

\section{Finitary homological algebra}

We use the symbol $fS$ to denote the set of all finite subsets of a given set $S$.     

An additive homomorphism
$\varphi:A\to  B$ between $G$-modules\footnote{Until Section \ref{novikov} the ground ring in this paper will be ${\Z}$.} is $G$-{\it finitary} (or just {\it finitary}) if it is captured by a $G$-map
$\Phi:A\to fB$, in the sense that $\varphi (a)\in \Phi (a)$ for every
$a\in A$. We call the $G$-map $\Phi$
is a $G$-{\it volley} (or just a {\it volley}) for the finitary map $\varphi$, and we say that $\varphi$ is a
{\it selection} from the volley $\Phi$.

Two volleys $\Phi :A\to fB$ and $\Psi
:B\to fC$ can be can be ``composed'' to give the volley $\Psi\Phi : A\to
fC$ defined by $\Psi\Phi(s) := \ds{\bigcup_{t\in \Phi(s)}}\Psi(t)$. A
$G$-map $\varphi : A\to B$ may be regarded as the $G$-volley which
assigns to every element $s\in S$ the singleton set $\{\varphi(s)\}$.
Hence $G$-volleys and $G$-homomorphisms can be composed in the above
sense.

Every $G$-homomorphism is, of course, $G$-finitary, but $G$-finitary
homomorphisms are much more general. Unlike a $G$-homomorphism, a
$G$-finitary map $\varphi:A\to B$ is not uniquely determined by its
values on a ${\mathbb Z}G$-generating set $X$ of $A$; however, the
possible values on $a = gx$ (where $g\in G$ and $x\in X)$ are restricted
to be in the finite set $\Phi(a) = g\Phi(x)\subseteq g\Phi(X)$. 

\begin{lemma}\label{L:3.1} 
If $\varphi : A\to B$ and $\psi : B \to C$ are $G$-finitary, so is the composition $\psi\varphi :
A\to C$.
\end{lemma}

\begin{proof}If $\varphi : A \to B$ and $\psi
: B\to C$ are selections from the volleys $\Phi : A\to fB$, $\Psi : B\to
fC$, respectively, then $\psi \varphi : A\to C$ is a selection from 
the composed volley $\Psi \Phi : A\to fC$.  
\end{proof}

Thus there is a $G$-finitary category of $G$-modules.

By a {\it based} free ${\Z}G$-module we mean a free (left) $G$-module $F$
with a specified basis.  We write $F=F_X$ when we wish to emphasize this
basis $X$.  The letter $Y$ will always stand for the induced ${\Z}$-basis
$Y = GX$.

\begin{example}\label{canonical} {\rm In this paper a $G$-volley will usually be given on based free
$G$-module $F_X$.  Indeed, if $B$ is an arbitrary $G$-module, every map
$\Phi :X\to fB$ extends to a {\it canonical volley} $\Phi
:F_X\to fB$ as follows:  On elements $y = gx$ of the ${\mathbb Z}$-basis $Y =
GX$, $\Phi$ is uniquely determined by $G$-equivariance:  $\Phi (y) :=
g\Phi(x)$; and for arbitrary elements $c = \ds{\sum_{y\in Y}}n_{y}y\in
F_X$, in the unique expansion, we put}
\begin{equation*}
\Phi (c) := \ds{\sum_{y\in Y}}n_{y}\Phi (y):=\{\ds{\sum_{y\in Y}}n_{y}b_{y}\mid b_{y}\in
\Phi (y),\text{ for all }y\in Y\}.
\end{equation*}
{\rm It is straighforward to check that $\Phi (gc)=g\Phi (c)$. We call 
$\Phi : F_X \to fB$ the {\it canonical $G$-volley induced by} $\Phi :X\to fB$.}  
\end{example}

With respect to the canonical $G$-volley, finitary homomorphisms are easy to construct on $F_X$: a finitary homomorphism $\varphi:F_{X}\to B$ can  be
given by first choosing  $\Phi(x)\in fB$ for each $x\in X$, and then
picking  $\varphi(gx)\in g\Phi(x)$ for all $(g,x)\in G\times X$.

\begin{examples}
\begin{enumerate}[1.]
\item If $H\leq G$ is a subgroup of finite index, and $A$, $B$ are $G$-modules then every
$H$-homomorphism $\varphi: A\to B$ is $G$-finitary.
\item If $N\leq G$ is a finite normal subgroup, and $A$ is a $G$-module then the additive
endomorphism of $A$ given by multiplication by 
$\lambda\in {\mathbb Z}N$ is $G$-finitary.
\end{enumerate}
\end{examples}

For more details, see \cite{BGe16}.

\subsection{Graded volleys and finitary chain maps}

In order to extend the results of \cite{BGe16} to higher dimensions we
need to know that a Comparison Theorem (Theorem \ref{L:5.1})
for projective resolutions is available for $G$-finitary homomorphisms.

\newline
Here are our standing notations and conventions: 
\newline
Until Section \ref{novikov}, ${\mathbf F}\thra A$ denotes a free 
${\Z}G$-resolution of the $G$-module $A$ by free $G$-modules $F_k$, i.e.
${\mathbf F}$ is graded as $\ds{\bigoplus_{k\geq 0}}F_k$ with boundary
morphism $\partial : {\mathbf F} \to {\mathbf F}$; in the final sections we allow more general ground rings than just ${\Z}$. Usually, the free $G$-modules $F_k$ will be finitely generated in dimensions $\leq n$. The truncation of ${\mathbf F}$ obtained by
setting $F_k =0$ when $k>n$ is the {\it $n$-skeleton} of ${\mathbf F}$
and is denoted by ${\mathbf F^{(n)}}$. Motivated by topology, we often refer to members of $\mathbf F$ as {\it chains} and to members of $Y$ as {\it cells}.

The augmentation morphism is $\epsilon:F_{0}\thra A$.  The corresponding
{\it augmented resolution} is the acyclic chain complex ${\mathbf
F}\thra A$.

Our free resolutions are {\it based}, meaning that each $F_k$ comes with
a specified free ${\Z}G$-basis $X_k$.  We write  $Y_k$ for the induced
$\Z$-basis $GX_k$.   We write $X$ and $Y$ for the unions of the $X_k$
and of the $Y_k$ respectively.

The based free resolution $\mathbf F\thra A$ is {\it admissible} if its basis
$X$ has the feature that $\partial x\neq 0$ for every $x\in X$, and
$\epsilon(x) \neq 0$ for every $x\in X_0$.  It is easy to replace
an arbitrary basis $X$ by a basis which makes $\mathbf F$ admissible
-- either by deleting basis elements $x$ with $\partial x= 0$, or by
replacing them by $x+x'$ if there is some $x'\in X$ with $\partial x'\neq
0$. {\it We will always assume that our based free resolutions are admissible.}

\begin{thm}[Finitary Comparison Theorem]\label{L:5.1} Let ${\mathbf
F}\thra A$ and ${\mathbf F'}\thra A'$ be admissible free resolutions
of the G-modules $A$ and $A'$. Then every $G$-finitary homomorphism
$f:A\to A'$ can be lifted to a $G$-finitary chain homomorphism 
$\varphi :{\mathbf F}\to {\mathbf F'}$ and any two such lifts are chain homotopic
by a $G$-finitary chain homotopy.  \end{thm}

In view of the straightforward construction of $G$-finitary maps on a
based free $G$-module (see Example \ref{canonical}) the first part of
the theorem --- the existence of a lift --- is obvious. The rest of this
section is concerned with the second assertion.

To get control of $G$-finitary chain maps and chain homotopies on free
resolutions ${\mathbf F}\to {\mathbf F}'$ we need volleys $\Phi : {\mathbf
F} \to f{\mathbf F}'$ of chain complexes.  We assume such volleys to be
graded of some degree $k$, i.e., $\Phi(F_n) \subeq F'_{n+k}$, for all
$n\geq 0$; but we do not require compatibility with the differentials.
However, degree 0 volleys $\Phi : {\mathbf F}\to f{\mathbf  F '}$ will
only be used in connection with chain maps, and hence {\it in degree
0 a selection will always be understood to be a chain-map-selection
from $\Phi$}.  And {\it graded volleys will always be understood to be degree
0 unless some other degree is specified}.  We say that the volley $\Phi :
{\mathbf F}\to f{\mathbf F'}$ {\it induces} the $G$-homomorphism $f : A\to A'$,
if $\epsilon'\Phi(c) = f\epsilon (c)$, for each $c\in F_0$. This implies
that all chain-map-selections from $\Phi$ induce $f$.

\begin{prop}\label{P:5.2} Let ${\bf F} \thra A$ and ${\bf F'}\thra A'$ be
admissible free resolutions of the $G$-modules $A$ and $A'$, and $\Phi
: {\bF} \to f{\bF'}$ a degree $0$ $G$-volley, inducing the zero map
$A\to A'$, i.e., with $\epsilon'\Phi(F_0) = 0$.  Then there is a
degree $1$ $G$-volley, $\Sigma : {\bf F} \to f{\bf F'}$, with the property that
every chain map $\varphi$ which is a selection from $\Phi$ is homotopic to $0$
by a homotopy which is a selection from $\Sigma$.
\end{prop}

\begin{proof}  We construct the volleys 
$\Sigma : {\bF}^{(n)} \to f{\bF'}^{(n+1)}$ by induction on $n$, starting with $n=0$.
As  $\text{im }\Phi_0\subeq \text{ker }\epsilon' = \text{im }\partial'_1$,
we can find, for each element $x$ of the $G$-basis $X_0$ of $F_0$,
a finite subset $\Sigma_0(x) \subeq {\bF'}_1$ with $\partial'\Sigma_0(x)
= \Phi_0(x)$.  This
defines a canonical $G$-volley $\Sigma_0 : F_{0}\to fF'_1$, and by
$G$-equivariance we have $\partial '\Sigma_0(y) = \Phi_0(y)$ for all
$y\in Y_0 =GX_0$.  Selections are determined by their restrictions to the
$\Z$-basis $Y_0 = GX_0$, so for each selection $\varphi$ from $\Phi_0$ there
is a selection $\sigma$ from $\Sigma_0$, with $\varphi = \partial'\sigma$.

Now we take $n\geq 1$, assuming the volley $\Sigma\mid {\bF}^{(n-1)} :
{\bF}^{(n-1)} \to f{\bF'}^{(n)}$ is already constructed, with the property
that every (chain-map) selection $\varphi\mid {\bF}^{(n-1)} : {\bF}^{(n-1)}
\to f{\bF'}^{(n-1)}$ is homotopic to zero by a homotopy which is a
chain-homotopy selection from $\Sigma\mid {\bF}^{(n-1}$.
For every chain-map-selection $\varphi$ from $\Phi\mid {\bF}^{(n)}$ there are
(possibly several) selections $\sigma$ from $\Sigma\mid {\bF}^{(n-1)}$
with $\varphi\mid {\bF}^{(n-1)} = \partial'\sigma + \sigma\partial$.
We consider all of them and use them to define, for each $c\in
F_n$,
\begin{equation}\label{E:5.1}
\begin{aligned}
\Gamma(c) &:= \{\varphi(c) - \sigma\partial(c)\mid\varphi \text { is 
a selection from } \Phi:F^{(n)}\to fF'^{(n)},
\text { and }\\ &\sigma \text{ is a selection from } 
\Sigma:F^{(n-1)}\to fF'^{(n)},\ \text{ with }\ \varphi\mid {\bF}^{(n-1)}=
\partial'\sigma + \sigma\partial\}.
\end{aligned}
\end{equation}
We claim that $\Gamma : F_n\to fF'_n$ is a $G$-volley.  To see this,
let $g\in G$, and let $\varphi$ and $\sigma$ be as in (\ref{E:5.1}).
Then\footnote{$g\varphi $ is defined by $g\varphi (a)=g\varphi(g^{-1}a)$.} 
$(g\varphi)\mid F^{(n-1)} = g(\varphi\mid F^{(n-1)}) = 
g(\partial'\sigma +
\sigma\partial) = \partial'(g\sigma) + (g\sigma)\partial$.  Moreover,
since $\varphi$ and $\sigma$ are selections from the $G$-volleys $\Phi\mid
{\bF}^{(n)}$ resp.\ $\Sigma\mid {\bF}^{(n-1)}$, so are $g\varphi$ and $g\sigma$.
Hence $(g(\varphi-\sigma\partial))(c) = 
g((\varphi-\sigma\partial)(g^{-1}(c))\in
\Gamma(c)$, for all $c$.  Replacing $c$ by $gc$ shows $g\Gamma(c)\subeq
\Gamma(gc)$, and replacing $g$ by $g^{-1}$ establishes the opposite
inclusion.  This shows that $\Gamma$ is $G$-equivariant; as $\Gamma $ is
given in terms of the $\Z$-homomorphisms $\varphi -\sigma \partial$, the
requirements for a volley hold.

Now we claim that $\partial'\Gamma = 0$.  Indeed, with $\varphi$ and 
$\sigma$ as in (\ref{E:5.1}), we find for all
$c\in F_n$,
\begin{equation*}
\partial'(\varphi(c)-\sigma\partial(c)) = \varphi\partial(c) - 
\partial'\sigma\partial (c) = \varphi\partial(c) - (\varphi\mid
{\bF}^{(n-1)} - \sigma\partial)\partial(c) = 0.
\end{equation*}

For every basis element $x\in X_n$, we can now choose a finite subset 
$\Sigma_n(X)\subeq {\bF'}_{n+1}$, with
$\partial'\Sigma_n(x) = \Gamma(x)$.  This defines a canonical 
$G$-volley $\Sigma_n : {\bF}_n \to 
f{\bF'}_{n+1}$, extending $\Sigma\mid {\bF}^{(n-1)}$ to $\Sigma\mid
{\bF}^{(n)}$.  By $G$-equivariance we have $\partial'\Sigma_n(y) = 
\Gamma(y)$, for all $y\in Y_n = GX_n$.

Let $\varphi$ be a chain-map-selection from $\Phi\mid {\bF}^{(n)}$.  By 
induction there is a selection $\sigma$ from
$\Sigma\mid {\bF}^{(n-1)}$, with $\varphi\mid {\bF}^{(n-1)} = 
\partial'\sigma + \sigma\partial$.  Then $\varphi(c) -
\sigma\partial(c) \in \Gamma(c)$, for every $c\in {\bF}_n$, hence 
$\gamma = (\varphi-\sigma\partial)\mid {\bF}_n$ is a
selection from $\Gamma : {\bF}_n \to f{\bF'}_n$.  Since selections are 
determined by their
restrictions to the $\Z$-basis $Y_n = GX_n$, and the selections of the 
canonical volley $\Sigma_n$ are free on
$Y_n$, we see that there is a selection $\sigma_n : F_n\to F'_{n+1}$ 
from $\Sigma_n$ with $\partial'\sigma_n =
\gamma = (\varphi - \sigma\partial)\mid {\bF}_n$.  Thus $\sigma_n$ 
extends $\sigma$ to a selection $\tau$ from 
$\Sigma\mid {\bF}^{(n)}$, with $\varphi = \partial'\tau + 
\tau\partial$, as asserted.
\end{proof}

\begin{cor}\label{C:5.3}
Let ${\bF} \thra A$ and ${\bF'} \thra A'$ be admissible free resolutions 
of the $G$-modules $A$ and $A'$, $\Phi$,
$\Psi : {\bF} \to f{\bF'}$ two degree-$0$-volleys.  Then there is 
a degree $1$ volley, $\Sigma : {\bF} \to
f{\bF'}$, with the property that any two chain-map-selections $\varphi$ 
from $\Phi$ and $\psi$ from $\Psi$, inducing the
same $G$-homomorphism $f : A\to A'$, are homotopic by a 
chain-homotopy-selection of $\Sigma$.
\end{cor}

\begin{proof} Consider the map $\Gamma : {\bF} \to f{\bF'}$, given by
\begin{equation*}
\Gamma(c) = \{\varphi(c) - \psi(c)\mid\varphi, \psi\ \text{ 
chain-map-selections of }\ \Phi,\ \text{resp. } \Psi, \
\text{ both lifting } f\}.
\end{equation*}
Then $\Gamma$ is a degree 0 volley inducing the zero map, and 
$\epsilon'\Gamma(F_0) = 0$.  Hence the
Corollary follows from Proposition \ref{P:5.2}.
\end{proof}

The Finitary Comparison Theorem \ref{L:5.1} follows from Corollary \ref{C:5.3}.


\section{Controlled based free resolutions}

\subsection{The control space}

Throughout the paper $(M,d)$ is a proper non-compact $CAT(0)$ metric space. The closed ball of radius $r$ centered at $b$ is denoted by $B_{r}(b)$.
We write $\rho : G\to \text{Isom}(M)$ for an action of the group $G$ on $M$ by
isometries. Unless specified, there are no further assumptions about
$\rho $; its orbits might be indiscrete, and it might not be cocompact.
Except in connection with the Openness Theorem in Section \ref{S:6.4}, the 
action $\rho $ is fixed throughout.

The boundary of $M$ at infinity, denoted by $\partial M$, is the set of
asymptoty classes of geodesic rays in $M$. It is assumed to carry the
(compact metrizable) cone topology, unless it is clear from context
that $\partial M$ is being considered with the Tits metric topology.
If $\gamma$ is a geodesic ray in $M$ determining $e\in \del M$ we
write $\gamma (\infty)=e$. (For given $e$ there is such a $\gamma$
with $\gamma (0)$ arbitrary.)  We write $\beta _{\gamma}:M\to {\mathbb
R}$ for the Busemann function \footnote{Our convention is that $\beta
_{\gamma}(x)$ goes to $+\infty$ as $x$ approaches $e$.} determined by $\gamma$ and we
write $HB_{(\gamma ,t)}$ for the (closed) horoball about $e$ determined
by the point $\gamma (t)$.  Usually we are interested in a difference of
the form  $\beta _{\gamma}(p)-\beta _{\gamma}(q)$ and such a difference
depends on $e$, rather than on the particular choice of $\gamma $ with
$\gamma(\infty )=e$.

\subsection{Controlled based free $G$-modules}\label{S:2}

The {\it support} $c\in F_{X}$, $\text {supp}(c) \subseteq Y$, is the set
of all $y\in Y=GX$ occurring in the unique expansion of $c$ over
${\mathbb Z}$. By a {\it control map} on $F$ we mean a $G$-map $h: F\to fM$ given\footnote{Recall that we write $fM$ for the $G$-set of all finite subsets of $M$.} by composing
the support function $\text{supp}:F\to fY$ with an arbitrary $G$-equivariant map $fY\to fM$,
where $h(0)$ is defined to be the empty set. Thus $h$ is uniquely given by
its restriction $h|:X\to fM$. We will always assume that
our control maps $h$ are {\it centerless} in the sense that $h(x)$
is non-empty for all $x\in X$ (and hence $h(c)\neq \emptyset$ for all
$0\neq c\in F$).

\subsection{Valuations on free modules}\label{S:2.3}

Let the point $e\in \partial M$ be determined by the geodesic ray $\gamma : [0,\infty) \to M$. Composition of the control map $h : F \to fM$, with the Busemann function $\beta_{\gamma} : M\to {\mathcal R}$ assigns to each element of $F$ a finite set of real numbers; taking minima defines the function  
\begin{equation}\label{E:2.2} 
v_\gamma := \min \beta_\gamma h : F \to {\mathcal R} \cup \{\infty.\}
\end{equation}
In particular $v_{\gamma}(c)=\infty$ if and only if  $c=0$.  

Following \cite{BRe88} we call $v_\gamma$ a {\it valuation} on $F$.

\begin{lemma}\label{L:2.3}
\begin{enumerate}[{\rm(i)}]
\item $v_\gamma(-c) = v_\gamma(c)$, for all $c\in F$,
\item $v_\gamma(c+c') \geq \min\{v_\gamma(c),v_\gamma(c')\}$, for all $c,c'\in F$, 
\item $v_\gamma(c) = v_{g\gamma }(gc)$, for all $c\in F$, $g\in G$.
\item If $c$ and $c'$ are non-zero then
$v_\gamma(c)-v_\gamma(c')$ depends only on the endpoint $\gamma(\infty)
= e$, not on the ray $\gamma$, and $|v_\gamma(c)-v_\gamma(c')| \leq d_{H}(h(c),h(c'))$.
\end{enumerate}
\hfill$\square$
\end{lemma} 
Once we have picked the control map $h : {\mathbf F} \to fM$, our
free resolution is equipped with a valuation $v_\gamma :=
\min\beta_\gamma h : {\mathbf F} \to {\mathcal R} \cup \{\infty\}$,
for each geodesic ray $\gamma : [0,\infty) \to M$.  On the augmented
resolution we have $v_\gamma(a) = \infty$, for each $a\in A$. 
\vskip 5pt
In our applications the finitely generated free module $F$ will be the $n$-skeleton ${\mathbf F}^{(n)}$ of a free resolution ${\mathbf F}\thra A$ where $A$ is a $G$-module. 

\begin{example} This example comes from topology. Take $A={\mathbb Z}$
and ${\bF}=C_{*}(\widetilde K)$, the integral simplicial chains in the
universal cover of a simplicial $K(G,1)$-complex $K$.  In this case ${\bf
F}$ comes with a canonical ${\Z}$-basis, the simplexes of $\widetilde K$,
and we can define a $G$-map ${\widehat h}:C_{*}(\widetilde K)\to fM$ on
each simplex $\sigma $ of $K$ by 
$$h(\sigma )={\widehat h}(\text{\{vertices of }\sigma\})$$.
\end{example}

\begin{rem} In this example we have $h(\del c)\subseteq h(c)$ for each
$c\in C_{*}(K)$. As a consequence we have $v_{\gamma}(\del c)\geq
v_{\gamma }(c)$, so that the chains with non-negative valuation form
a subcomplex. One could mimic that in the general situation by first
choosing $h(x)\neq \emptyset $ for each $x\in X_0$, and then defining
$h(x)$ on the higher skeleta by $h(x):=h(\del x)$. However, there is no
need for this in general and so our control maps can ignore the boundary
aspect of the resolution.  
\end{rem}

\section{Controlling homomorphisms over $M$}\label{S:3.3}

\subsection{Controlling homomorphisms on free modules}

Let the based free modules $F_{X}$ and $F'_{X'}$ be endowed with control maps $h$ and $h'$ mapping to $M$.   We want to measure how far, in terms of the metric $d$ on $M$, an additive homomorphism ${\varphi} : F \to F'$ moves
the members of $F$.  We define the {\it norm} of $\varphi $ by 
\begin{equation}\label{E:3.4} 
||\varphi|| :=\inf\{r\geq 0\mid h'(\varphi (c))\subseteq N_{r}(h(c));c\in F\}\in {\mathbb R}\cup \{\infty\}
\end{equation} 
{\it the shift function towards} $e$, $\text{sh}_{\varphi,e}
: F\to {\mathbb R} \cup \{\infty\}$, by 
\begin{equation}\label{E:3.5}
\text{sh}_{\varphi,e}(c) := v'_\gamma(\varphi(c)) - v_\gamma(c) \in {\mathbb
R}\cup \{\infty\},\ c\in F, 
\end{equation} 
and the {\it guaranteed shift towards} $e$ by, 
\begin{equation}\label{E:3.6}
\text{gsh}_e(\varphi) := \inf\{\text{sh}_{\varphi,e}(c)\mid c\in F\}.
\end{equation}

We call a $\Z$-submodule $L\leq F_X$ {\it cellular} if it is generated by
$L\cap Y$.  Sometimes $L$ will be given, and we will be interested in the norm or guaranteed shift of $\varphi \mid L$.  To have information for that case we include $L$ in the next lemmas.

\begin{lemma}\label{L:3.2} Let $\varphi :L\to F'$ be the restriction to $L$ of an 
additive homomorphism $F\to F'$.
\begin{enumerate}[{\rm(i)}]
\item  $\text{\rm sh}_{\varphi,e}(y) \geq -||\varphi||$, for all $y\in L\cap Y$; hence {\rm gsh}$_e(\varphi) \geq -||\varphi||$.
\item $||g\varphi|| = ||\varphi||$, {\rm sh}$_{g\varphi,ge} = \text{\rm sh}_{\varphi,e}$ and {\rm gsh}$_{ge}(g\varphi) =
\text{\rm gsh}_e(\varphi)$, for all $g\in G$. 
\end{enumerate}
\hfill$\square$
\end{lemma}

\begin{lemma}\label{L:3.3}
Let $\varphi : F\to F'$ and $\psi : F'\to F''$ be two additive endomorphisms, and let $K \leq F$ and $L\leq F'$ be cellular $\Z$-submodules with
$\psi(K)\subeq L$.  Then 
\begin{equation*}
\text{\rm gsh}_e(\varphi |L\circ \psi |K) \geq \text{\rm gsh}_e(\varphi |K) + \text{\rm gsh}_e(\psi |L).
\end{equation*}
In particular, 
\begin{equation*}
\text{\rm gsh}_e(\varphi^k)\geq k\cdot \text{\rm gsh}_e(\varphi), \text{ for all natural numbers }k.
\end{equation*}
\end{lemma}

\begin{proof} We use Lemma \ref{L:3.2}(ii).
\begin{equation*}
\begin{aligned}
\text{gsh}_e(\varphi \mid \circ \psi \mid K) &= \inf_{c\in K}\{v''_\gamma(\varphi\psi(c)) - v_\gamma(c)\}\\
&= \inf_{c\in K}\{v''_\gamma(\varphi \psi(c)) - v'_\gamma(\psi(c)) + v'_\gamma(\psi(c)) - v_\gamma(c)\}\\
&\geq \inf_{c\in K}\{v''_\gamma(\varphi \psi(c)) - v'_\gamma(\psi(c))\} 
+ \inf_{c\in K}\{v'_\gamma(\psi(c)) -v_\gamma(c)\}\\
&\geq \inf_{b\in L}\{v''_\gamma(\varphi(b)) -v'_\gamma(b)) + \inf_{c\in K}(v'_\gamma(\psi(c))-v_\gamma(c)\}\\
&= \text{gsh}_e(\varphi |L) + \text{gsh}_e(\psi |K).
\end{aligned}
\end{equation*}
\end{proof}

We say that an additive endomorphism $\varphi : F\to F$ {\it pushes $L$ towards} $e\in \partial M$, and we call $\varphi $ a {\it push towards} $e$, if the guaranteed shift of $\varphi |L$ towards $e$ is positive; i.e., gsh$_e(\varphi|L) > 0$.

\subsection{Pushing submodules towards limit points of orbits in 
$\partial M$}\label{S:3.4}

When $\varphi$ is $G$-finitary then $||\varphi||$ and  $\text{\rm gsh}_e(\varphi)$ are finite. When $\Phi$ is a finite $G$-volley we call the number
$||\Phi||:=\text{ inf}\{r\geq 0\mid h'(\Phi (c))\subseteq N_{r}(h(c))\}$ the
{\it norm} of $\Phi$. Then $||\varphi||\leq ||\Phi||$ for all selections $\varphi$ from $\Phi $.  

In this subsection we assume that the cellular submodule  $L\leq F$
is in fact a ${\Z}G$-submodule.  It will then be generated, as a
${\Z}G$-module, by $X' = L\cap X\subeq X$. Since $\varphi$ pushes the $G$-submodule $L$ towards $e$ with guaranteed
shift $\delta$, the $G$-translate $g\varphi$ of $\varphi$
pushes $L$ with the same guaranteed shift $\delta$ towards $ge$.
In the special case when $\varphi |L$ is $G$-finitary we can do better:
given any $\hat e\in\text{cl}(Ge)$, the closure of the $G$-orbit $Ge\in \partial M$,
we can still construct endomorphisms pushing towards $\hat e$ which are
``approximated'' by $G$-translates of $\varphi\mid L$:

\begin{thm}\label{P:3.4}  Let $L\leq F$ be an cellular $G$-submodule
of $F=F_X$ and let $\varphi : L \to F$ be a selection from the finite $G$-volley
$\Phi :L\to fF$ with $\text{\rm gsh}_e(\varphi) = \delta > 0$.  Then for every
endpoint $\hat e\in \text{\rm cl}(Ge)$ there is a selection $\psi : L\to F$
of $\Phi$ with $\text{\rm gsh}_{\hat e}(\psi) \geq \delta/2$.  In fact, this can
be done so that on each finitely generated $\Z$-submodule $L'\subeq L$,
$\psi$ coincides with  some $G$-translate $g\varphi$.  \end{thm}
The proof can be found in \cite{BGe16}.


\section{The Dynamical Invariants $\Sigmacirc (M;A)$}\label{S:5}

Let ${\bF} \thra A$ be a controlled based free resolution of
the $G$-module $A$, with finitely generated $n$-skeleton ${\bF}^{(n)}$. For $n\geq 0$ we define the $n$th {\it dynamical invariant} of the 
pair $(M,A)$ to be
\begin{equation*}
\Sigmacirc(M;A):= \{e\in \partial M \mid \text{ there is a } 
G\text{-finitary chain map inducing id}_{A} \text{ pushing }{\bF}^{(n)} 
\text{ towards } e\}
\end{equation*}

\begin{prop} [Invariance] \label{P:5.4} , Let $e\in\partial M$. The existence of a $G$-finitary chain map $\varphi :{\bF}^{(n)} \to {\bF}^{(n)}$ inducing {\rm id}$_A$ and pushing ${\bF}^{(n)}$ towards $e$ is independent of the choice of the resolution 
${\bF}\thra A$ and of the control map $h : {{\bF}} \to fM$. In other words, $\Sigmacirc(M;A)$ is well defined.  
\end{prop}

\begin{proof} Let ${\bF'}\thra A$ be a second such resolution with 
finitely generated $n$-skeleton.  The identity map id$_A$ can be
lifted to $G$-chain homomorphisms $\alpha : {\bF} \to {\bF'}$, $\beta : 
{\bF'} \to {\bF}$ which are chain homotopy
inverse to one another.  Assume there exists a $G$-finitary push 
$\varphi : {\bF}^{(n)} \to {\bF}^{(n)}$ lifting id$_A$.
Then $\alpha\varphi\beta : {\bF'}^{(n)} \to {\bF'}^{(n)}$ is a $G$ 
finitary chain endomorphism lifting id$_A$.  By Lemmas
\ref{L:3.2} and \ref{L:3.3}, gsh$_{e}(\alpha\varphi^k\beta) \geq - 
||\alpha|| +
k\cdot \text { gsh}_{e}(\varphi ^k) - ||\beta||$. If we choose $k$ large enough
to ensure that $k\cdot \text { gsh}_{e}(\varphi ) > ||\alpha|| + 
||\beta||$, the map $\alpha \varphi^k\beta : {\bF'}^{(n)} \to 
{\bF'}^{(n)}$ becomes a $G$-finitary push towards $e$.  This shows that 
the existence of a finitary push towards $e$ lifting id$_A$ is 
independent of the particular free resolution.  Independence of the 
control map is proved as a special case: take ${\bF}={\bF'}$, $\alpha $ 
an automorphism, and $\beta $ the inverse of $\alpha $.
\end{proof}

The set $\Sigmacirc(M;A)$ is invariant under the topological 
action of $G$ on $\partial M$ induced by the isometric action of $G$ 
on $M$.  For inductions we define 
${^{0}\Sigma}{^{-1}}(M;A) = \partial M$.

A slight adaptation of Theorem \ref{P:3.4} yields a closure result

\begin{thm}\label{T:5.5}
The $G$-set $\Sigmacirc(M;A)$ contains the closure of each of its orbits.
\end{thm}

\begin{proof}  Let $\varphi : {\bF}^{(n)} \to {\bF}^{(n)}$ be a 
$G$-finitary chain map pushing ${\bF}^{(n)}$ towards $e\in 
\Sigmacirc(M;A)$.  The proof of Proposition \ref{P:3.4} constructs 
a $G$-finitary map $\psi : {\bF}^{(n)}\to {\bF}^{(n)}$ pushing towards an 
arbitrary  point of the closure of $Ge$ with the property that for 
every finitely generated
$\Z$-submodule $L\leq {\bF}^{(n)}$ there is some element $g\in G$, with 
$\psi\mid L = (g\varphi)\mid L$. It remains to show that $\psi$ is a 
chain map. But since $g\varphi$
is a chain map, so is $\psi\mid L$, for each $L$.  This suffices.
\end{proof}


\section{The Geometric Invariants $\Sigma ^{n}(M;A)$}\label{S:4}

Here we define the $n$th Geometric Invariant $\Sigma ^{n}(M;A)$. It
is the strict homological analog of the ``homotopical" invariant $\Sigma
^{n}(\rho)$ described in \cite{BGe03}.

\subsection{Controlled acyclicity}\label{control}

Recall that once we have picked the control map $h : {\mathbf F} \to fM$, our
free resolution is equipped with valuation $v_\gamma :=
\min\beta_\gamma h : {\mathbf F} \to {\mathcal R} \cup \{\infty\}$,
for each geodesic ray $\gamma : [0,\infty) \to M$.  On the augmented
resolution we have $v_\gamma(a) = \infty$, for each $a\in A$.

Let $n\geq 0$ and let $\gamma (\infty )=e$. We say that the augmented 
controlled based free resolution ${\mathbf F} \thra A$ is {\it 
controlled $(n-1)$-acyclic over} $e\in \partial M$, in short $CA^{n-1}$
over $e$, if for every real number $s$ there is a {\it lag} $\lambda(s)
\geq 0$, with $s-\lambda(s) \to +\infty$ as $s\to +\infty$, and
such that, for $-1\leq i\leq n-1$, every $i$-cycle $z\in {\mathbf F}$
with $v_\gamma(z) \geq s$ is the boundary, $z =\partial c$, of an
$(i+1)$-chain $c\in {\mathbf F}$ with $v_\gamma(c) \geq 
s-\lambda(s)$. When $n=0$ this is to be understood as a condition on 
the augmented chain complex.

\subsection{Invariance}\label{S:4.3}

Let $({\mathbf F},\partial)$ and $({\mathbf F}',\partial')$ be 
controlled based free
resolutions of the $G$-module $A$, with $\epsilon$, $\epsilon'$ the
corresponding augmentation maps.

\begin{prop}[Invariance]\label{T:4.2} Let $\mathbf F$ and ${\mathbf 
F}'$ have finitely generated $n$-skeleta. If ${\mathbf F}'$ is 
$CA^{n-1}$ over $e$, so is $\mathbf F$.
\end{prop}

\begin{proof}  Consider $\mathbf F$ and ${\mathbf F}'$ with bases $X$,
$X'$ and corresponding control maps $h$ and $h'$.  There are
$G$-chain-homomorphisms $\varphi : {\mathbf F}\to {\mathbf F}'$ and $\psi
: {\mathbf F}'\to {\mathbf F}$, inducing the identity on $A$, and a
chain-homotopy $\sigma : {\mathbf F} \to {\mathbf F}$, with $\psi\varphi =
1+\partial \sigma + \sigma \partial$.
For $0\leq i<n$ let $z$ be
an $i$-cycle of $\mathbf F$ with $v_\gamma(z)\geq s$.
We denote by $\varphi$, $\psi$, $\sigma$, and $\partial$ the chain
maps, the homotopy, and the boundary as above, but restricted
to the $n$-skeleta.  By Lemma \ref{L:3.2}, $v'_\gamma(\varphi(z)) \geq
v_\gamma(z)-||\varphi||$.  As ${\mathbf F}'$ is $CA^{n-1}$ over $e$
there is a chain $c'$ in ${\mathbf F}'$, with $\partial 'c' =\varphi(z)$,
and
\begin{equation*} v'_\gamma(c')\geq v'_\gamma(\varphi (z)) - \lambda 
(v'_\gamma(\varphi (z))) \geq
v_\gamma(z) - ||\varphi|| -\lambda (v'_\gamma(\varphi (z))),\text{ 
where }\lambda \text{ is independent of }c'
\end{equation*}
Put $c''= \psi(c') -\sigma(z)$.  Then
\begin{equation*}
\partial c'' = \partial\psi(c') -\partial\sigma(z) = \psi\partial'(c') -
\partial \sigma(z) =\psi\varphi(z) - \partial \sigma(z) = z+\sigma\partial
z = z,
\end{equation*}
and we have,
\begin{equation*}
\begin{aligned}
v_\gamma(c'') &\geq \min\{v_\gamma(\psi(c')),v_\gamma(\sigma(z))\}\\
&\geq \min\{v'_\gamma(c') - ||\psi||,v_\gamma(z)-||\sigma||\}\\
&\geq \min\{v_\gamma(z)-||\varphi||-\lambda(v'_\gamma(\varphi (z))) -
||\psi||,v_\gamma(z)-||\sigma||\}
\end{aligned}
\end{equation*}
proving that ${\mathbf F}$ is $CA^{n-1}$ over $e$.
\end{proof}

\subsection{The Geometric Invariants}\label{S:4.4}

We can now introduce the geometric (or $\Sigma$-) invariants of the
pair $(M,A)$ where the $G$-module $A$ is of type $FP_n$. Choosing
a free resolution with finitely generated $n$-skeleton ${\mathbf
F}\thra A$, we define 

\begin{equation*} \Sigma^n(M;A) := \{e\in \partial
M\mid {\mathbf F}\thra A\text{ is } CA^k\text{ over }e\text{ for all
}k\text{ with }-1\leq k< n\}.  
\end{equation*}

By Theorem \ref{T:4.2}
this is an invariant of $(M;A)$, i.e. $(n-1)$-acyclicity over $e$ is
independent of the choice of free resolution ${\mathbf F}\thra A$ such
that ${\mathbf F^{(n)}}$ is finitely generated, and of the choice of
control map.  In particular, this subset of $\partial M$ is invariant
under the topological action of $G$ on $\partial M$ induced by the
isometric action $\rho $ of $G$ on $M$. For proofs using induction on $n$
we define $\Sigma^{-1}(M;A):= \partial M$.

We will use the phrase ``$e\in \Sigma^n(M;A)$ {\it with constant 
lag} $\lambda \in {\mathbb R}$'' if the function $\lambda(s)$ in the 
definition of $CA^{n-1}$ in Section \ref{control} can be taken to be 
the constant $\lambda$.  For trivial reasons, all members of 
$\Sigma^0(M;A)$ have constant lag.


\section{Characterization of $\Sigmacirc(M;A)$ in terms of
$\Sigma^{n}(M;A)$}\label{S:6}

In this section we characterize $\Sigmacirc(M;A)$ as a specific subset of
$\Sigma^n(M;A)\}$ (Theorem \ref{T:6.1}), and we give conditions
under which $\Sigmacirc(M;A)=\Sigma^n(M;A)$ (Theorem
\ref{T:6.4}).

\subsection{Statement of the Theorem}

\begin{thm}[Characterization Theorem]\label{T:6.1}
For each $G$-module $A$ of type $FP_n$, $n\geq 0$, we have
\begin{equation*}
\Sigmacirc(M;A) = \{e\in \partial M\mid \text{\rm cl}(Ge)\subeq 
\Sigma^n(M;A)\}.
\end{equation*}
\end{thm}

This shows that $\Sigmacirc(M;A)$ is determined by $\Sigma^n(M;A)\}$.

\begin{rems}
\begin{itemize}
\item In the special case where all the endpoints $e\in \partial M$
are fixed under the induced action of $G$ on $\partial M$, Theorem 
\ref{T:6.1} implies
$\Sigma ^{n}(M;A)=\Sigmacirc(M;A)$. Hence
Theorem \ref{T:6.1} is a direct generalization of the various ``$\Sigma$-Criteria'' found in 
\cite[Proposition 2.1]{BS80}, \cite[Proposition 2.1]{BNS87}, \cite[Theorem
C]{BRe88}. These were the main technical tools
in all previous stages of $\Sigma $-theory.  In all those cases, the 
action of $G$ was by translations on a Euclidean space, so that all 
end points were fixed.
\item The homotopy version of Theorem \ref{T:6.1} was proved in \cite{BGe03}.
\end{itemize}
\end{rems}

The proof of Theorem \ref{T:6.1} will be given in two steps:  the
inclusion $\subeq$ follows from Proposition \ref{P:6.2} together with
Theorem \ref{T:5.5}.  The other inclusion $\supeq$ follows from the
(stronger) Theorem \ref{T:6.4}, below.

\subsection{From pushing skeletons to constant lag}\label{S:6.2}

We start by proving that $\Sigmacirc(M;A)\subeq \Sigma^n(M;A)$; 
while doing that we will also collect important information on the 
lag.  More precisely, we prove

\begin{prop}\label{P:6.2}
Let ${\bF} \thra A$ be a controlled based free resolution with finitely 
generated $n$-skeleton, $n\geq 0$.  Let $\varphi : {\bF}^{(n)} \to {\bF}^{(n)}$
be a $G$-finitary chain endomorphism inducing $\text{\rm id}_A$, and 
let $\sigma : \varphi \sim \text{\rm id}_F$ be a
$G$-finitary chain homotopy.  If $\varphi$ pushes ${\bF}^{(n)}$ towards 
$e$ then the following hold:
\begin{enumerate}[{\rm (i)}]
\item $e\in \Sigma^n(M;A)$, with constant lag $||\sigma||$, and
\item If ${\bF}^{(n+1)}$ is finitely generated and $e$ is in 
$\Sigma^{n+1}(M;A)$ then it is so with constant lag
$||\sigma||$.
\end{enumerate}
\end{prop}

\begin{proof}  (i) Let $z$ be a $j$-cycle of $\bF$, with $j\leq n-1$, 
and let $c$ be a chain, with $z = \partial
c$.  Using $\varphi-\text{id}_F = \sigma\partial + \partial\sigma$, 
and writing $\varphi^k$ for the $k$-th iterate of
$\varphi$, we find that $\tau := (\varphi^k+\varphi^{k-1} + \cdots 
\varphi+1)\sigma$ is a $G$-finitary homotopy $\varphi^{k+1}\sim 
\text{id}_F$.  So
\begin{equation*}
\begin{aligned}
\varphi^{k+1}(c) &= c+\partial \tau(c) + \tau(\partial c),\ \text{ hence}\\
&z = \partial c\\
&= \partial \varphi^{k+1}(c) - \partial\tau(\partial c)\\
&= \partial c',\ \text{ where }\ c' = \varphi^{k+1}(c)-\tau(\partial c).
\end{aligned}
\end{equation*}
Let $\mu > 0$ be a guaranteed shift of $\varphi$ towards $e$.  Using 
Lemma \ref{L:3.3} we find
\begin{equation*}
v_\gamma(\varphi^k(c)) \geq v_\gamma(c)+k\cdot\mu;
\end{equation*}
hence we may choose $k$ so large that $v_\gamma(\varphi^{k+1}(c))\geq 
v_\gamma(z)$.  Then
\begin{equation*}
v_\gamma(c')\geq \min\{v_\gamma(\varphi^{(k+1)}(c)),
v_\gamma(\tau(z))\}\geq \min(v_\gamma(z),v_\gamma(\tau(z))).
\end{equation*}
But
\begin{equation*}
\begin{aligned}
v_\gamma(\tau(z)) &= v_\gamma((\varphi^k+\varphi^{k-1}+\hdots 
\varphi+1)\sigma(z))\\
&\geq \min\{v_\gamma(\varphi^p\sigma(z))\mid 0\leq p\leq k\}\\
&\geq \min\{v_\gamma(\sigma(z)) + p\cdot \mu\mid 0\leq p\leq k\}\ 
\text{ by Lemma \ref{L:3.3}},\\
&=v_\gamma(\sigma(z)).
\end{aligned}
\end{equation*}
So
\begin{equation*}
v_\gamma(c') \geq \min(v_\gamma(z),v_\gamma(\sigma(z))).
\end{equation*}
By Lemma \ref{L:3.2}(iii), $v_{\gamma}(\sigma (z)) \geq 
v_\gamma(z)-||\sigma||$; hence $v_\gamma(c')\geq
v_\gamma(z) - ||\sigma||$, showing that $e\in \Sigma^n(M;A)$, with 
lag $||\sigma||$.  This proves statement
(i).

(ii)  The assumption $e\in \Sigma^{n+1}(M;A)$ asserts that $\bF$ 
is $CA^n$ over $e$.  Thus there is a lag
$\lambda(s)$ with the property that every $n$-cycle $z$ with 
$v_\gamma(z)\geq s$ is the boundary of some
$(n+1)$-chain $c$ with $v_\gamma(c)\geq s-\lambda(s)$, and 
$(s-\lambda(s))\to \infty$, as $s\to\infty$.  We fix
an $n$-cycle $z$ and apply the lag condition to the sequence of 
$n$-cycles $\varphi^k(z)$.  Put $s_k :=
v_\gamma(\varphi^k(z))$.  As before let $\mu > 0$ be a guaranteed 
shift of $\varphi$ towards $e$. By Lemma \ref{L:3.3}, $s_k\geq 
v_\gamma(z) + k\cdot \mu$, hence $s_k\to\infty$.  Thus
we can choose $k$ so that $s_{k+1}-\lambda(s_{k+1}) > s_0 
=v_\gamma(z)$.  It follows that there is some $(n+1)$-chain $c'$ with
$\partial c' = \varphi^{k+1}(z)$ and $v_\gamma(c') \geq 
s_{k+1}-\lambda(s_{k+1}) >v_\gamma(z)$.

Much as in Part (i), and using this new choice of $k$, we put $\tau := 
\sigma(\varphi^k+\varphi^{k-1}+\hdots \varphi+1)$.
This is a $G$-finitary homotopy $\varphi^{k+1}\sim 
\text{id}_F$.  Since $z$ is a cycle we have
$\varphi^{k+1}(z) = z+\partial\tau(z)$.  Writing $c'' = c' - \tau(z)$ 
we have $\partial c'' = z$ and
\begin{equation*}
\begin{aligned}
v_\gamma(c'') &\geq \min\{v_\gamma(c'),v_\gamma(\tau(z)\}\\
&\geq \min\{v_\gamma(z),v_\gamma(\tau(z))\}.
\end{aligned}
\end{equation*}
Now,
\begin{equation*}
\begin{aligned}
v_\gamma(\tau(z)) &= 
v_\gamma(\sigma(\varphi^k+\varphi^{k-1}+\hdots\varphi+1)(z))\\
&\geq \min\{v_\gamma(\sigma\varphi^p(z))\mid 0 \leq p\leq k\}\\
&\geq \min\{v_\gamma(\varphi^p(z)) - ||\sigma||\mid 0\leq p\leq k\}, 
\ \text{ by Lemma \ref{L:3.3}(ii)}\\
&\geq \min\{v_\gamma(z) + p\cdot\mu - ||\sigma||\mid 0\leq p\leq k\}\\
&= v_\gamma(z) - ||\sigma||.
\end{aligned}
\end{equation*}
Thus $v_\gamma(c'') \geq v_\gamma(z) - ||\sigma||$ and we conclude 
that $||\sigma||$ is a constant lag for the $CA^n$-property of $\bF$ over $e$.
\end{proof}

If $\Phi : {\bF}^{(n)} \to f{\bF}^{(n)}$ is a $G$-volley then,
by the Finitary Homotopy Lemma \ref{L:5.1}, there is a finite degree $-1$
volley $\Sigma : {\bF}^{(n)} \to f{\bF}^{(n+1)}$ with the property that
every selection $\varphi$ of $\Phi$ is chain contractible by a selection
$\sigma$ of $\Sigma$.  The norm of $\sigma$
has an upper bound depending only on $\Sigma$. This yields the following
uniform version of Proposition \ref{P:6.2}.

\begin{cor}\label{C:6.3}
Let $E\subeq \partial M$ be a set of endpoints.  Let $\Phi : 
{\bF}^{(n)}\to f{\bF}^{(n)}$ a $G$-volley inducing $\text{\rm 
id}_A$, with the property that each $e\in E$ admits a (chain map) 
selection $\varphi_e : {\bF}^{(n)}\to {\bF}^{(n)}$ of $\Phi$ pushing the 
$n$-skeleton towards $e$.  Then the following hold
\begin{enumerate}[{\rm (i)}]
\item  $E\subeq \Sigma^n(M;A)$, with uniform constant lag, (i.e., 
the same constant lag for all $e\in E$);
\item if ${\bF}^{(n+1)}$ is finitely generated and $E\subeq 
\Sigma^{n+1}(M;A)$ then it is so with uniform
constant lag.
\end{enumerate}
\end{cor}
\begin{proof} We apply \ref{P:6.2} to each $\varphi_e$; with $\Sigma$ 
as above, there is a chain contraction
$\sigma_e : \varphi_{e}\sim \text{id}$ which is a selection from 
$\Sigma$.  The lags are therefore independent of $e$.
\end{proof}

\subsection{Closed $G$-invariant subsets of $\partial M$}\label{S:6.3}

The next theorem gives conditions under which $\Sigmacirc(M;A)$ 
and $\Sigma^n(M;A)$ agree. In particular, it contains Theorem \ref{T:6.1}.

\begin{thm}\label{T:6.4}
Let ${\bF}\thra A$ be a controlled based free resolution with finitely 
generated $n$-skeleton. The following conditions are equivalent for a 
closed $G$-invariant set of endpoints $E\subeq \partial
M$.
\begin{enumerate}[{\rm (i)}]
\item  $E\subeq \Sigma^n(M;A)$,
\item $E\subeq \Sigmacirc(M;A)$,
\item there is a uniform constant $\lambda$ such that for all $e\in 
E$, $e\in \Sigma^n(M;A)$ with constant lag
$\lambda$,
\item there is a uniform constant $\nu > 0$ and a $G$-volley 
$\Phi : {\bF}^{(n)}\to f{\bF}^{(n)}$ inducing id$_A$ with the
property that for each $e\in E$ there is a chain map selection 
$\varphi_e$ from $\Phi$ with $\text{\rm gsh}_e(\varphi_e)
\geq \nu$.
\end{enumerate}
\end{thm}

\begin{proof} All implications except (i) $\Rightarrow$ (ii) and (i) 
$\Rightarrow$ (iv) have been proved, and the
latter is the stronger of these two, so we prove (i) $\Rightarrow$ 
(iv). The proof is by induction on $n$.

When $n = 0$ there is no difference between (ii) and (iv).  Assume 
$E\subeq \Sigma^0(M;A)$.  For each $x\in
X_0$ and $e\in E$ we choose $\bar c(e,x)\in F_0$ such that 
$\epsilon(\bar c(e,x)) = \epsilon(x)$ and
$v_\gamma(\bar c(e,x)) -v_\gamma(x) > 0$, where $\gamma(\infty) = e$.
If this inequality holds for $e$ and $x$, then it also holds for $e'$ 
and $x$ when $e'$ lies in a suitably small
neighborhood of $e$.  Since $E$ is compact there is a finite subset
$E_f\subeq E$ such that for all $e\in E$ there is some $e'\in E_f$ 
such that $v_\gamma(\bar c(e',x))
-v_\gamma(x) > 0$ when $\gamma(\infty) = e$.  For every $e\in E$ we 
choose such an $e'$ and define $c(e,x) := \bar
c(e',x)$. Thus $\ds{\inf_{e\in E}}\{v_\gamma (c(e,x))-v_\gamma(x)\} > 
0$.  Define $\Psi(x) = \{c(e,x)\mid e\in
E\}$, a finite set of $0$-chains.  For $y = gx$, define $\Psi(y) = 
g\Psi(x)$.  Then the induced $\Psi : F_0\to fF_0$ is a 
$G$-volley inducing id$_A$.

For $e\in E$ and $y = gx$ define an additive endomorphism $\psi_e : F_0 
\to F_0$ by $\psi_e(y) := gc(g^{-1}e,x)$; this
makes sense because $E$ is $G$-invariant.  Then $\epsilon\psi_e = 
\epsilon$, and
\begin{equation*}
\begin{aligned}
v_\gamma(\psi_e(y)) - v_\gamma(y) &= v_\gamma(gc(g^{-1}e,x)) -v_\gamma(gx))\\
&= v_{g^{-1}\gamma}(c(g^{-1}e,x)) - v_{g^{-1}\gamma}(x).
\end{aligned}
\end{equation*}
Thus $\ds{\inf_{e\in E}}\{\text{gsh}_e(\psi_e)\} > 0$.  This proves 
$E\subeq \overset \circ \Sigma{^0}(M;A)$
and finishes the case $n=0$.

Now we assume $n\geq 1$. We are given that $E\subeq 
\Sigma^n(M;A)$, and by induction we may assume
the statement of (iv) when $n$ is replaced by $n-1$.  We also know 
that $F_n$ is finitely generated.  So, by
Corollary \ref{C:6.3}(ii) we conclude that $E\subeq \Sigma^n(M;A)$ 
with a uniform constant lag $\lambda \geq
0$.  We have a $G$-volley $\Phi :{\bF}^{(n-1)} \to f{\bF}^{(n-1)}$ 
inducing id$_A$ such that for each $e\in
$ there is a chain map selection $\varphi_e$ from $\Phi$ with 
gsh$_e(\varphi_e) \geq \nu > 0$.  Hence, by Lemma \ref{L:3.3}, for any positive
integer $k$ we have gsh$_e(\varphi^k_e)\geq k\nu$.  We may choose $k$ so that 
$k\nu \geq \lambda +||\partial \mid {\bF}^{(n)}|| + \delta$ where 
$\delta > 0$ is arbitrary.  The endomorphisms $\varphi^k_e$ are 
selections from the finite
$G$-volley $\Phi^k : {\bF}^{(n-1)} \to f{\bF}^{(n-1)}$.

For $x\in X_n$ define $\Pi(x) := \{g^{-1}\varphi^k_e(g\partial x)\mid 
g\in G,e\in E\}$.  This is a finite set of
cycles, hence of boundaries.  For each $(e,p) \in E\x \Pi(x)$ we 
choose $c(e,x,p)\in F_n$ such that
$\partial(c(e,x,p)) = p$ and
\begin{equation*}
v_\gamma(c(e,x,p)) > v_\gamma(p)-\lambda\geq v_\gamma(p)-k\nu + 
||\partial\mid {\bF}^{(n)}||+\delta.
\end{equation*}
Just as in the case $n=0$, above, the compactness of $E$ allows us to 
make our choices $c(e,x,p)$ from a finite
set $\Psi(x)\subeq F_n$.  Putting $\Psi(y) := g\Psi(x)$, when $y = 
gx\in Y_n$, we get a $G$-volley $\Psi :
{\bF}^{(n)}\to f{\bF}^{(n)}$ extending $\Phi^k$.  Let $\psi_e : F_n\to 
F_n$ be the homomorphism defined by
\begin{equation*}
\psi_e(y) := gc(g^{-1}e,x,g^{-1}\varphi_e\partial y).
\end{equation*}
Then $\partial\psi_e = \varphi^k_e\partial$, so $\psi_e$ is a chain 
map extending $\varphi^k_e$.  When $\gamma(\infty)
= e$ we have
\begin{equation*}
\begin{aligned}
v_\gamma(\psi_e(y)) &= v_\gamma(gc(g^{-1}e,x,g^{-1}\varphi^k_e(g\partial x)))\\
&= v_{g^{-1}\gamma}(c(g^{-1}e,x,g^{-1}\varphi^k_e(\partial y)))\\
&> v_{g^{-1}\gamma}(g^{-1}\varphi^k_e(\partial y)) - \nu + 
||\partial\mid {\bF}^{(n)}|| + \delta\\
&= v_\gamma(\varphi^k_e(\partial y)) - \nu + ||\partial\mid 
{\bF}^{(n)}|| + \delta\\
&\geq v_\gamma(\partial y) + \text{gsh}_e(\varphi^k_e)-\nu + 
||\partial\mid {\bF}^{(n)}|| + \delta\\
&\geq v_\gamma(y) + \delta.
\end{aligned}
\end{equation*}
So gsh$_e(\psi_e)\geq \delta$. Thus (iv) holds for $n$, and the 
induction is complete.
\end{proof}



\section{The meaning of $\Sigma ^{n}(M;A)=\partial M$}\label{S:7}

From now on we assume the module $A$ is non-zero. In this section and the
next we study the meaning of $\Sigma^{n}(M;A)=\partial M$.  We note
that, by Theorem \ref{T:6.4}, the statements $\Sigma^n(M;A)=\partial
M$ and $\Sigmacirc(M;A)=\partial M$ are equivalent.  Our first
goal is Theorem \ref{T:7.8}, which explains how
these properties are also equivalent to what we will call ``controlled
$(n-1)$-acyclicity over $M$".

\subsection{Controlled acyclicity over points $b\in M$}\label{S:7.1}

Controlled acyclicity over a point $b\in M$ is analogous to controlled
acyclicity over an endpoint $e\in \partial M$.  The role of the valuation
on the augmented controlled based free resolution ${\bF}\thra A$ is 
played by the function $D_b :
{\bF}\to {\mathbb R}_{\geq 0}$ defined by  $D_b(c) :=\max \{d(p,b)\mid
p\in h(c)\}$ when $c\neq 0$ and $D_b(c) = 0$ when $c=0$.  We extend
$D_b$ to the module $A$ by $D_b(a)=0$ for all $a\in A$. We say 
${\bF}\thra A$ is {\it
controlled} $(n-1)$-{\it acyclic over} $b\in M$, in short $CA^{n-1}$ over
$b$, if there is a {\it lag} function $\lambda : {\mathbb R}_{\geq 0}
\to {\mathbb R}_{\geq 0}$, such that, for any any $-1\leq
i\leq n-1$, each (augmented) $i$-cycle is the boundary, $z = \partial
c$ of some $(i+1)$-chain $c$ satisfying
\begin{equation}\label{E:8.1}
D_b(c) \leq D_b(z) + \lambda(D_b(z)).
\end{equation}

\begin{prop}\label{P:8.0}
\begin{enumerate}[\rm (i)]
\item If $\lambda
_{R}:[R,\infty)\to {\mathbb R}_{\geq 0}$ satisfies the inequality 
(\ref{E:8.1}) when $D_{b}(z)\geq R$, and if $\lambda _R$ is extended
to $[0,\infty )$ by defining $\lambda _{R}(s)=\lambda _{R}(R)+R-s$ when
$0\leq s\leq R$, then the extended $\lambda _R$ satisfies that inequality
for all $z$, and hence is a lag function in the above sense.
\item If
$d(b,b')=\delta $ and $\lambda _R$ is a lag function with respect to $b$,
then $\lambda _{R} +2\delta$ is a lag function with respect to $b'$.
\hfill $\square$
\end{enumerate}
\end{prop}

Proposition \ref{P:8.0}(i) implies that if in the definition of
``$CA^{n-1}$ over $b$" the lag function can be chosen to be constant on
some interval $[R,\infty )$ then it can be chosen to be constant over all.
Proposition \ref{P:8.0}(ii) implies that if the action $\rho $ is
cocompact (so that there is a fundamental domain of finite diameter) and
if we have $CA^{n-1}$ over some point  $b$, then we have $CA^{n-1}$ over
all points $b$ with the same lag function being applicable everywhere.
We refer to this  as a {\it uniform lag}.  We say that ${\bF}\thra A$ is
$CA^{n-1}$ {\it over } $M$ if it is $CA^{n-1}$ over each $b\in M$ with
a uniform lag function.

In particular, ``$CA^{-1}$ over $M$" means that there
is a number $\lambda $ $(=\lambda(0))$ such that for every $a\in A$ and
$b\in M$ there is a 0-chain $c$ with $\epsilon (c)=a$ and 
$h(c)\subseteq B_{\lambda }(b)$. Since $A$ is non-zero this implies 
that the action $\rho $ is cocompact.

\begin{prop}\label{P:7.0}
Let ${\bF} \thra A$ be a controlled based free resolution with finitely 
generated $n$-skeleton.  Let $\varphi : {\bF}^{(n)} \to {\bF}^{(n)}$
be a $G$-finitary chain endomorphism inducing $\text{\rm id}_A$, and 
let $\sigma : \varphi \sim \text{\rm id}_F$ be a
$G$-finitary chain homotopy.  If $\varphi$ pushes ${\bF}^{(n)}$ towards 
some point of $M$ then the following hold:

\begin{enumerate}[{\rm (i)}]
\item $\rho $ is cocompact, and  ${\bF}$ is $CA^{n-1}$ over $M$
with constant lag.
\item If ${\bF}^{(n+1)}$ is finitely generated and ${\bF}$ is $CA^{n}$
over $M$, then it is so with constant lag
\end{enumerate}
\end{prop}
\begin{proof} The proof is entirely analogous to that of Proposition
\ref{P:6.2}, and the details are therefore omitted. Instead of pushing
towards some $e\in \partial M$ we are now pushing towards $b$.  Of course,
there exists $R\geq 0$ such that those portions of our chains already 
over the ball of radius $R$ about
$b$ do not make progress towards $b$,
but Proposition \ref{P:8.0} implies that this makes no difference.
\end{proof}

{\bf Remark:} In fact there is a radius $R$ such that for any $b\in M$
the lag outside $B_{R}(b)$ can be $||\sigma ||$.

\subsection{Bounded support and cocompactness of $\rho$}\label{S:7.2}

We first consider the special case $n=0$. We say the module $A$ has 
{\it bounded
support over} $M$ if there is a bounded subset $B\subseteq M$
with the property that for each $a\in A$
there exists $c\in F_0$ with $\epsilon(c) = a$ and $h(c)\subseteq B$. 
By the triangle inequality, it
is easy to see that this property is independent of the point $b\in M$,
though the number $r$ varies with $b$.  It is also independent of the
choice of $F$ and $h$.

\begin{thm}\label{P:7.1}
Let $A$ be a finitely generated non-zero $G$-module.  The following 
are equivalent:
\begin{enumerate}[{\rm(i)}]
\item $\Sigma^{0}(M;A)=\partial M$;
\item $(M;A)$ is $CA^{-1}$ over $M$;
\item $\rho$ is cocompact and $A$ has bounded support over $M$.
\end{enumerate}
\end{thm}

The example of $SL_{2}({\mathbb Z})$ acting on the hyperbolic plane 
-- here $A$ is the trivial module ${\mathbb Z}$ -- shows that 
``having bounded support" does not imply ``cocompact".

\begin{proof} The equivalence of (ii) and (iii) is clear.  We spelled
out the meaning of ``$CA^{-1}$ over $M$", and (using that notation)
there is a ball of radius $\lambda $ inside any horoball; so (ii)
implies (i). That (i) implies cocompactness follows from Proposition
6.6 of \cite{BGe16}. The remaining item, the fact that (i) implies bounded
support, requires some work\footnote{Theorem \ref{P:7.1} appears with
full proof in our paper \cite{BGe16}. Since it is the $n=0$ case of a bigger
theorem, Theorem \ref{T:7.8}, some of the methods get used later. So,
in the next section we sketch the proof of Theorem \ref{P:7.1} referring
to \cite{BGe16} for some detailed proofs.}. \end{proof}

\subsection{Shifting towards base points}\label{S:7.3}

Let $F = F_X$ be a finitely generated based free $G$-module, choose 
a base point $b\in M$ and
consider the canonical control map $h : F\to M$ with respect to $b$ 
and the basis $X$. Let
$\varphi : F\to F$ be an additive endomorphism, and let $L\leq F$ be an 
cellular $\Z$-submodule of $F$.  The {\it shift function of} 
$\varphi\mid L$ {\it towards} $b\in M$ measures the loss of distance 
from $b\in M$; it is denoted by sh$_{\varphi,b} : L\cap Y\to {\mathbb 
R}$, and is defined by
\begin{equation}\label{E:7.1}
\text{sh}_{\varphi,b}(y) := D_b(y) - D_b(\varphi(y)) \in {\mathbb 
R}\cup \{\infty\},\quad y\in L\cap Y.
\end{equation}

The notion of guaranteed shift towards $b\in M$ is more subtle than 
the corresponding notion for  endpoints $e\in \partial M$ because if 
elements are already too close to $b$ it may not be possible to push 
them any closer.  Therefore we have to restrict attention to elements 
$y\in L\cap Y$ with $h(y)$ outside some ball centered at $b$.  When 
$\alpha \in {\mathbb R}$ and $R\geq 0$, the pair $(\alpha,R)$ {\it 
defines a guaranteed shift} for $\varphi\mid L$ if 
sh$_{\varphi,b}(y)\geq \alpha$ whenever $y\in L\cap Y$ and $D_b(y) > 
R$.  The {\it (almost) guaranteed shift
of} $\varphi$ {\it on} $L$ is
\begin{equation*}
\text{gsh}_b(\varphi\mid L) := \sup\{\alpha\mid \text{ for some } R,\ 
(\alpha,R) \text{ defines a guaranteed shift for
}\varphi\mid L\}.
\end{equation*}

We call a number $R$ occurring in such a pair $(\alpha, R)$ an {\it 
event radius} for $\varphi$.

For a proof of the following lemma see Section 9 of \cite{BGe16}:

\begin{lemma}\label{L:7.2}
\begin{enumerate}[{\rm (i)}]
\item $-||\varphi\mid L|| \leq \text{\rm gsh}_b(\varphi\mid L) \leq 
||\varphi \mid L||$.
\item If $\psi : F\to F$, and $K$ is an cellular submodule with 
$\psi(K) \subeq L$ then
\begin{equation*}
\text{\rm gsh}_b(\varphi\mid L\circ \psi\mid K) \geq \text{\rm 
gsh}_b(\varphi\mid L) + \text{\rm gsh}_b(\psi\mid K).
\end{equation*}
\end{enumerate}
\end{lemma}

We note that when $\varphi$ is $G$-finitary $||\varphi\mid L|| <\infty$ and
gsh$_{b}(\varphi\mid L)$ is attained.  If gsh$_b(\varphi\mid L) > 0$ 
we say that
$\varphi$ {\it pushes} $L$ {\it towards} $b\in M$.

\begin{cor}\label{C:7.3}
If $\varphi$ in Lemma \ref{L:7.2}(ii) pushes $L$ towards $b$, and 
$\varphi(L)\subeq L$ then $\varphi^k\circ \psi$ pushes
$L$ towards $b$ when $k > \dfrac{-\text{\rm gsh}_b(\psi\mid 
K)}{\text{\rm gsh}_b(\varphi\mid L)}$.  In fact,
{\rm gsh}$_b(\varphi^k\circ\psi\mid L) > \eta$ when $k > 
\dfrac{\eta-\text{\rm gsh}_b(\psi\mid K)}
{\text{\rm gsh}_b(\varphi\mid L)}$.
\end{cor}

The $CAT(0)$ metric space $M$ is {\it almost geodesically complete} if
there is a number $\mu\geq 0$ such that for any $b$ and $b'\in M$ there
is a geodesic ray $\gamma$ starting at $b$ and passing within $\mu $ of $b'$.
An example lacking this property is the half line $[0,\infty)$.
It is a theorem in \cite{GO07} that whenever the isometry group of $M$
acts cocompactly then $M$ is almost geodesically complete.

\begin{lemma}\label{L:7.4}
Let $M$ be a proper $CAT(0)$ space, and let $r\geq 0$, $\delta > 0$.
\begin{enumerate}[{\rm (i)}]
\item If $(b,e) \in M\x \partial M$ is given then, by choosing $p = 
p(b,e)\in M$ sufficiently far out on the geodesic ray $\gamma$ from 
$b$ to $e$, we can achieve
\begin{equation}\label{E:7.3}
|(\beta_\gamma(q) - \beta_\gamma(p)) - (d(p,b) - d(q,b))| < \delta 
\text{ when }d(p,q)\leq r.
\end{equation}
\item Assume $M$ is almost geodesically complete and let $\mu$ be as 
in that definition. There is a number $R=R(r,\delta)$ such that 
(\ref{E:7.3}) holds whenever $d(p,b)\geq R$, $\gamma $ is a geodesic 
ray starting at $b$ and passing within $\mu $ of $p$, and $d(p,q)\leq r$.
\end{enumerate}
\end{lemma}

\begin{proof}  (i) is immediate from the definition of horoballs and
Busemann functions.  The proof of (ii) is contained in the proof of
Theorem 15.3 of \cite{BGe03}.  \end{proof}

The next proposition is proved in Section 9 of \cite{BGe16}:

\begin{prop}\label{P:7.6}
Let $M$ be almost geodesically complete proper $CAT(0)$ space. 
The following are equivalent for a $G$-volley $\Phi:{\bf F}\to f{\bf F}$:
\begin{enumerate}[{\rm (i)}]
\item $\forall e\in \partial M$ $\Phi$ admits a selection pushing ${\bf F}$ 
towards $e$ which induces $\text{\rm id}_A$.
\item $\forall b\in M$ $\Phi$ admits a selection pushing ${\bf F}$ towards $b$ which induces $\text{\rm id}_A$.
\end{enumerate}
\end{prop}

\begin{proof} ({\it Completion of proof of Theorem \ref{P:7.1}}): We
assume (i) and we know that this implies cocompactness.  Hence, by the
theorem of \cite{GO07} mentioned above, it follows that $M$ is almost
geodesically complete. So for any $b\in M$ Theorem \ref{T:6.4} and 
Proposition \ref{P:7.6}
give us a $G$-volley $\varphi $ having a selection $\varphi $ with
${\text{\rm gsh}_b(\varphi)}>0$.  Let $(\alpha , R)$ define a guaranteed
shift for $\varphi$, where $\alpha >0$.  For any $a\in A$ there is a 0-chain
$c$ mapped by $\epsilon $ to $a$ such that
$$D_{b}(\varphi (c))\leq {\text{\rm max }}\{R+||\varphi ||,
D_{b}(c)-\alpha \}.$$
By iterating $\varphi
$ we can thus move $c$ over $M$ to a new $c'$ such that $\epsilon
(c')=a$ and $h(c')$ lies over the ball centered at $b$ with radius
$R+||\varphi ||$, a number independent of $a$. \end{proof}

\subsection{The higher dimensional case}

\begin{prop}\label{P:7.7}
Let ${\bF}\thra A$ be a based free resolution over $M$ with finitely 
generated $n$-skeleton, where $n\geq 1$, and
let $h : {\bF}\to fM$ be a canonical control function at $b\in M$ with 
respect to the basis $X$. Let $\varphi : {\bF}^{(n-1)} \to
{\bF}^{(n-1)}$ be a $G$-finitary chain map inducing the identity on 
$A$, with $\varphi$ pushing ${\bF}^{(n-1)}$ towards $b$. If 
$\Sigma^n(M;A)=\partial M$ then some iterate $\varphi^k$ of 
$\varphi$ admits a $G$-finitary chain
map extension $\psi : {\bF}^{(n)}\to {\bF}^{(n)}$ pushing ${\bF}^{(n)}$ 
towards $b\in M$.  The guaranteed shift, the event
radius of $\psi$, and the number $k$ depend on $\varphi$ and the 
uniform lag, but are independent of $b\in M$.
\end{prop}

\begin{proof}  Let $\Phi : {\bF}^{(n-1)}\to f{\bF}^{(n-1)}$ be a 
$G$-volley from which  $\varphi$ is a selection.  By Theorem \ref{T:6.4}
we know that ${\bF}^{(n)}$ is $CA^{n-1}$ with respect to every endpoint 
$e\in \partial M$, with uniform constant lag
$\lambda$.  Now, Corollary \ref{C:7.3} asserts that for suitable $k$, 
$\varphi^k\partial : {\bF}^{(n)} \to {\bF}^{(n)}$
pushes all of ${\bF}^{(n)}$ towards $b\in M$; in fact, by choosing $k$ 
sufficiently large, we can achieve
\begin{equation}\label{E:7.5}
\text{gsh}_b(\varphi^k\partial) \geq \lambda +3\delta, \text{ where 
}\delta > 0\text{ is arbitrary}.
\end{equation}
We aim to extend the $k$-th iterate $\Phi^k$ to a $G$-volley on the 
$n$-skeleton by our usual compactness argument. It suffices to define this
extension on the finite $G$-basis $X_n$.  Let $x\in X_n$.  We 
observe that the set of chains
\begin{equation*}
\Pi(x) := \{g^{-1}\varphi^k(g\partial x)\mid g\in G\}
\end{equation*}
lies in $g^{-1}\Phi^k(\partial gx)$ and hence is finite, with 
$\partial \Pi(x) = 0$.  For each pair
$(e',p) \in \partial M\x \Pi(x)$ we can choose an $n$-chain 
$c(e',p)\in F_n$, with $\partial c(e',p) = p$, and
\begin{equation}\label{E:7.6}
v_{\gamma'}(c(e',p)) > v_{\gamma'}(p)-\lambda,
\end{equation}
where $\gamma '{(\infty)}=e'$.

Fixing $p$ for a moment, we observe that if (\ref{E:7.6}) holds for 
some $e'\in \partial M$ then there is a
neighborhood $N(e')$ of $e$ in $\partial M$ such that
\begin{equation*}
v_{\gamma''}(c(e',p)) > v_{\gamma''}(p) - \lambda\ \text{ for all } 
e'' = \gamma''(\infty) \in N(e').
\end{equation*}
Since $\partial M$ is compact, finitely many of these neighborhoods 
cover $\partial M$.  This shows that we can
improve on the choice of the $n$-chains $c(e',p)$ as follows: we can 
find a {\it finite} set of $n$-chains, which
we denote $\Psi(x)\subeq F_n$, with the property that for each 
$(e',p)\in \partial M\x \Pi(x)$ there is some
$c(e',p)\in\Psi(x)$, with $\partial c(e',p)=p$ which satisfies the inequality
(\ref{E:7.6}).  Putting $\Psi(gx) := g\Psi(x)$ defines
a $G$-volley $\Psi : F_n \to fF_n$.

Since $\Sigma^n(M;A)=\partial M$, Theorem \ref{P:7.1} and the theorem
from \cite{GO07} cited above imply that $M$ is almost geodesically
complete.  Let $\mu  > 0$ be the number given by the definition of
``almost geodesically complete''.  For every $y\in gx\in Y_n$ we choose
an endpoint $e =e(y)$, with the property that the geodesic ray $\gamma
= \gamma_y$ from $\gamma(0) = h(y)$ to $e =\gamma(\infty)$, passes the
point $b\in M$ at distance $<\mu$.  We put $\psi(y) := gc(g^{-1}e(y),
g^{-1}\varphi^k(g\partial x)) \in g\Psi(x)$, noting that $\partial\psi(gx) =
g\partial c(g^{-1}e(y),g^{-1}\varphi^k(g\partial x)) = \varphi^k(\partial gx)$,
hence $\partial\psi = \varphi^k\partial$, as required.

It remains to show that $\psi$ pushes towards $b\in M$.  Let $R_1 =
R(||\Phi^k\partial||,\delta)$ be the radius given by Lemma \ref{L:7.4}
where $\delta >0$ is arbitrary.  Then Lemma \ref{L:7.4} yields
\begin{equation}\label{E:7.7}
|\text{sh}_{\varphi^k\partial,e(y)}(y) - 
\text{sh}_{\varphi^k\partial,b}(y)|< \delta, \text{ for all }y = gx\in
Y_n\text{ with } D_b(y)) \geq R_1.
\end{equation}

Now, we take $R_2$ to be an event radius for $\varphi^k\partial : 
{\bF}^{(n)}\to {\bF}^{(n)}$.  For every $y = gx\in
Y_n$, with $D_b(y)\geq R_2$ and $\gamma = \gamma_y$ the ray from 
$h(y)$ to $e=e(y)$, we find
\begin{equation*}
\begin{aligned}
v_\gamma(\psi(y))-v_\gamma(y) &= 
v_\gamma(gc(g^{-1}e,g^{-1}\varphi^k(\partial y)))-v_\gamma(y) &\\
&= v_{g^{-1}\gamma}(c(g^{-1}e,g^{-1}\varphi^k(\partial 
y)))-v_\gamma(y), &\text{by Lemma \ref{L:2.3}},\\
&> v_{g^{-1}\gamma}(g^{-1}\varphi^k(\partial y)) - 
v_\gamma(y)-\lambda, &\text{by (\ref{E:7.6})},\\
&= v_\gamma(\varphi^k\partial y) - v_\gamma(y) -\lambda, &\text{by 
Lemma \ref{L:2.3}},\\
&= \text{sh}_{\varphi^k\partial,e}(y)-\lambda, &\\
&\geq \text{sh}_{\varphi^k\partial,b}(y)-\delta-\lambda, &\text{by 
(\ref{E:7.7})}\\
&\geq \text{gsh}_b(\varphi^k\partial) -\delta -\lambda, &\\
&\geq 2\delta, &\text{by (\ref{E:7.5})}.
\end{aligned}
\end{equation*}
Hence sh$_{\psi,e(y)}(y) \geq 2\delta$, and therefore, by Lemma 
\ref{L:7.4} there exists $R_3(||\Psi||,\delta)$
such that sh$_{\psi,b}(y)\geq \delta$ when $D_b(y) > R_3$.  Thus we 
find gsh$_{b}(\psi) > 0$, i.e. $\psi$ pushes $F_n$ towards $b$.
\end{proof}

\begin{thm}\label{T:7.8}
Let $M$ be a proper $CAT(0)$ space.  Let ${\bF} \thra A$ be an 
augmented $G$-free resolution with finitely generated $n$-skeleton, 
and $h :{\bF}\to fM$ a control map.  The following are equivalent:
\begin{enumerate}[{\rm(i)}]
\item $\Sigma^n(M;A)=\partial M$
\item There are positive numbers $(R,\alpha)$ with the property that 
for every $b\in M$ there is $G$-finitary chain map $\varphi_b 
:{\bF}^{(n)}\to {\bF}^{(n)}$, inducing $\text{\rm id}_A$ and pushing all 
of the $n$-skeleton towards $b\in M$, with guaranteed shift $\alpha$ 
and event radius $R$.
\item ${\bF}\thra A$ is controlled $(n-1)$-acyclic over $M$.
\item ${\bF}\thra A$ is controlled $(n-1)$-acyclic over $M$ with a constant lag.
\end{enumerate}
\end{thm}

\begin{proof} We begin with (i) $\Rightarrow$ (ii). From Theorem
\ref{T:6.4} we know that there is a volley $\Phi:{{\bf} F}^{(0)}\to
f{\bf F}^{(0)}$ inducing the identity on $A$ and satisfying (i) of
Proposition \ref{P:7.6}. Thus Proposition \ref{P:7.6} applied to this
volley gives a chain map $\varphi_{b}:{\bf F}^{(0)}\to {\bf F}^{(0)}$ pushing
${\bf F}^{(0)}$ towards $b$. From the proof of Proposition  \ref{P:7.6}
we see that the event radius and the guaranteed shift of $\varphi _b$
depend only on $\Phi$ and not on $b$.   This starts an induction,
the inductive step being given by Proposition \ref{P:7.7}.

(ii) $\Rightarrow$ (iv): This follows from Part (i) of Proposition \ref{P:7.0}.

(iv) $\Rightarrow$ (i): Let the geodesic ray $\gamma$ define an end point
$e$ and let $HB_{\gamma (t)}$ be a horoball.  A cycle over this horoball
also lies over some ball centered along $\gamma $.  If the constant
lag in the hypothesis is $\lambda$ then this cycle bounds a chain over
the ball obtained by increasing the previous ball's radius by $\lambda$.
That ball lies in the horoball $HB_{\gamma (t)-\lambda }$.

(iii) $\Rightarrow$ (iv): This follows from Part (ii) of Proposition
\ref{P:7.0}. (iv) $\Rightarrow$ (iii) is trivial.
\end{proof}

For $b\in M$ we write $G_b$ for the subgroup of $G$ fixing $b$. For 
any other $b'$ the group $G_{b'}$ is commensurable with $G_b$.

\begin{cor}\label{T:7.9} Assume that the action $\rho$ is cocompact and that its 
orbits are discrete subsets of $M$. Then $\Sigma^n(M;A)=\partial 
M$ if and only if $A$ has type $FP_n$ as an 
$G_b$-module, where $b\in M$.
\end{cor}

\begin{proof}  Filter $\bf F$ by $h^{-1}(fB_{n}(b))$, where $b\in M$,
$n\geq 1$, and the notation means the largest $\Z$-subcomplex mapped by $h$
into $fB_{n}(b)$.  These subcomplexes are $G_b$-invariant, and because
the orbits are discrete these subcomplexes are finitely generated modulo
$G_b$ in dimensions $\leq n$. According to (an obvious adaptation of)
Theorem 2.2 of \cite{Bn87}, $A$ has type $FP_n$ as an $G_b$-module if
and only if ${\bF} \thra A$ is $CA^{n-1}$ over $M$.  By Theorem \ref{T:7.8}
the Corollary follows.  \end{proof}

A variant is:

\begin{cor} Assume that the orbits of the action $\rho $ are discrete 
subsets of $M$ and that the group $\rho (G)$ acts properly 
discontinuously and cocompactly (aka ``geometrically") on $M$. Then 
$\Sigma^n(M;A)=\partial M$ if and only if $A$ has type $FP_n$ as a
$\Z${\rm [ker($\rho $)]}-module.
\end{cor}
\begin{proof} The hypothesis implies that $N$ and $G_b$ are commensurable.
\end{proof}


\section{Dispensing with lags}\label{S:8}

We continue to assume the module $A$ is non-zero. Let ${\bF}\thra A$ be 
a controlled based free resolution with finitely generated 
$n$-skeleton with control map $h : {\bF}\to fM$. In this section we 
show that when $\Sigma^n(M;A) = \partial M$ we can replace $\bF$ by another
such resolution $\bF'$ and define a control map $h' : {\bF'} \to fM$ so 
that the pre-images under $h'$ of
horoballs and of large balls are $(n-1)$-acyclic.  In short, we can 
reduce the lags to zero.

We begin with the horoball case, and with $n=0$.  When $e\in 
\Sigma^0(M;A)$ then for each
$x\in X_0$ and $\nu > 0$, there exists $c\in F_0$ with $\epsilon(c) = 
\epsilon(x)$ and
\begin{equation}\label{E:7.8}
v_\gamma(c)-v_\gamma(x) > \nu
\end{equation}
where $\gamma$ is a geodesic ray with $\gamma(\infty) = e$. We write 
${\bF}^{(\gamma,t)}$ for the subcomplex generated by $\{y\in 
Y|v_{\gamma }(y)\geq t\}$.

\begin{prop}\label{P:7.9}
  When $e\in \Sigma^0(M;A)$ the augmentation map $\epsilon$ takes 
$F_0^{(\gamma,t)}$ onto $A$.
\hfill $\square$
\end{prop}

This is equivalent to saying that ${\bF}$ is $CA^{-1}$ over $e$.

Next, assume $\Sigma^0(M;A) = \partial M$.  Just as in the $n=0$ 
proof of Theorem \ref{T:6.4}, a compactness
argument shows that for each $x\in X_0$ there is a finite set 
$\Phi(x)$ so that for every $e\in \partial M$
(\ref{E:7.8}) holds for some $c\in \Phi(x)$.  The resulting function 
$X_0 \to fF_0$ defines a $G$-volley
$\Phi : F_0 \to fF_0$.  We now alter $F_1$ and $F_2$ by performing 
``elementary expansions''.  For each $c\in
\Phi(x)$ we choose $d\in F_1$ with $\partial d = x -c$. We add new 
generators $\xi$ to $X_1$ and $\eta$ to $X_2$,
defining $\partial\xi = x-c$ and $\partial \eta= d -\xi$.  We extend 
$h$ to $h'$ by $h'(\xi) = h(x)\cup h(c)$ and $h'(\eta)
= h(x) \cup h(c) \cup h(d)$.  The resulting enlarged chain complex 
${\bF'} \thra A$ is again a free resolution with
finitely generated $n$-skeleton.  We note that $\xi = \xi(x,c)$ with 
$x\in X_0$ and $c\in \Phi(x)$.  Define
$\Sigma(x) := \{\xi(x,c)\mid c\in \Phi(x)\}$.  This function $X_0 \to 
fF_1$ defines a finite degree 1
$G$-volley $\Sigma : F_0 \to fF_1$.  The definition of $h'(\xi(x,c))$ ensures
\begin{equation}\label{E:7.9}
v_\gamma(\xi(x,c)) = v_\gamma(x).
\end{equation}

\begin{lemma}\label{L:7.10}
For each $e\in \partial M$ there are selections $\varphi_e$ from 
$\Phi$ and $\sigma_e$ from $\Sigma$ so that, for all
$y\in GX_0$, $\partial \sigma_e(y) = y-\varphi_e(y)$, 
$v_\gamma(\varphi_e(y)) - v_\gamma(y) > \nu$, and
$v_\gamma(\sigma_e(y)) = v_\gamma(y)$.
\end{lemma}

\begin{proof}  Fix $e\in \partial M$. Let $y = gx$.  Using 
(\ref{E:7.8}) and (\ref{E:7.9}) pick $c$ and $\xi(x,c)$
so that
\begin{enumerate}
\item $v_{g^{-1}\gamma}(c) - v_{g^{-1}\gamma}(x) > \nu$,
\item $v_{g^{-1}\gamma}(\xi) = v_{g^{-1}\gamma}(x)$ and
\item $\partial \xi = x-c$.
\end{enumerate}
Then $v_\gamma(gc) - v_\gamma(y) > \nu$ and $v_\gamma(g\xi) = 
v_\gamma(y)$.  Define
$\varphi_e(y) = gc$ and $\sigma_e(y) = g\xi$.  Then $\partial 
\sigma_e(y) = y-\varphi_e(y)$.
\end{proof}

Lemma \ref{L:7.10} provides a $G$-finitary push $\varphi_e$ towards 
each $e$ and a $G$-finitary chain homotopy
$\sigma_e : \text{id} \sim\varphi_e$.  We think of these chain 
homotopies as ``monotone'' because they have the
property that $v_\gamma(\sigma_e(c)) \geq v_\gamma(c)$ for all chains $c$.

\begin{prop}\label{P:7.11}
Assume $\Sigma^0(M;A)=\partial M$. Let $e\in \Sigma^1(M;A)$ and 
let $h'$ be the extended control map on $\bF'$.  Then the resolution ${\bF'}
\thra A$ is $CA^0$ over $e$ with zero lag.  Equivalently, for any 
$t$, ${\bF'}^{(t)}$ is $0$-acyclic.
\end{prop}

\begin{proof} Writing $e =\gamma(\infty)$ there is a lag 
$\lambda(e,t)$ as in the definition of $CA^0$ in Section
\ref{S:4}.  Let $z$ be a 0-cycle over $HB_{\gamma,t}$.  For $k$ a 
positive integer we consider the chain homotopy
$\bar\sigma_{e,k} := \sigma_e(\varphi^k_e + \varphi^{k-1}_e + \cdots 
+ \varphi_e + 1)$ as in the proof of Proposition
\ref{P:6.2}(ii).  If $k$ is large enough, $\varphi^k(z)$ bounds over 
$HB_{\gamma,t}$ and because $\sigma_e$ is
monotone $\bar\sigma_{e,k}$ provides a homology over $HB_{\gamma,t}$ 
between $z$ and $\varphi^k(z)$.  Thus $z$
bounds over $HB_{\gamma,t}$.
\end{proof}

Next, we repeat for $n=1$ what we have just done for $n=0$.  Assuming 
$\Sigma^1(M;A) = \partial M$ we extend
the $G$-volleys $\Phi$ and $\Sigma$ by defining $\Phi : F_1 
\to fF_1$ and $\Sigma : F_1\to fF_2$, adding
new generators in dimensions 2 and 3.  The only difference is that in 
the analog of Lemma \ref{L:7.10} we will have
$\partial\sigma_e(y) = y -\varphi_e(y) - \sigma_e\partial y$. The 
pattern for higher $n$ is now clear.  We have
proved:

\begin{prop}\label{T:7.12}
When $\Sigma^{n-1}(M;A) = \partial M$ there is a controlled based 
free resolution ${\bF} \thra A$ with finitely generated $n$-skeleton which is
$CA^{n-1}$ with zero lag over every $e\in \Sigma^n(M;A)$.  \hfill $\square$
\end{prop}

Essentially the same proof gives:

\begin{prop}\label{T:7.13}
Let $E$ be a closed $G$-invariant subset of $\partial M$.  When 
$\Sigma^{n-1}(M;A)\supeq E$ there is a
resolution ${\bF}\thra A$ with finitely generated $n$-skeleton which is 
$CA^{n-1}$ with zero lag over every $e\in E\cap
\Sigma^n(M;A)$.  \hfill$\square$
\end{prop}

Proposition \ref{T:7.13} applies, in particular, to a singleton set 
$\{e\}$ where $e$ is a fixed point of the $G$-action on $\partial
M$.  For example, in the Euclidean case, where $G$ acts by 
translations, every point of the boundary is fixed by $G$, and this
recovers Theorem 4.2 of \cite{BRe88}.

A straightforward adaptation of the proof of Proposition \ref{T:7.12}
gives the following addition to Theorem \ref{T:7.8}:

\begin{thm}\label{T:7.14}
$\Sigma^n(M;A) = \partial M$ if and only if there is a controlled 
based free resolution ${\bF} \thra A$ with finitely generated 
$n$-skeleton and a radius $R$
with the property that $h^{-1}(fB)$ is $(n-1)$-acyclic when $B$ is 
any ball of radius $\geq R$ (or any horoball,
for that matter).
\end{thm}


\section{Openness theorems}\label{S:6.4}

When $E\subseteq \partial M$ we write ${\mathcal
R}_{E}:=\text{Hom}(G,\text{Isom}(M,E))$, the set of all isometric
actions of $G$ on $M$ which leave $E$ invariant. We endow the sets
$\text{Isom}(M,E)$ and ${\mathcal R}_{E}$ with the compact-open topology
\footnote{$G$ is of course discrete.}. The boundary $\del M$ carries the cone
topology.

In this section, when we discuss a particular action $\rho \in {\mathcal
R}_{E}$ we will write $_{\rho}M$ rather than $M$.  We choose a base point
$b\in M$.  The canonical control map $h^{\rho}: F\to f(_{\rho}M)$
takes the $\Z$-generator $gx$ to the singleton set $\{\rho (g)b\}\subseteq M$. 

\begin{thm}[$\text{Openness Theorem}$]\label{T:6.5} Let  $E$ be a
compact subset of $\partial M$, and let $\rho \in {\mathcal R}_{E}$
be such that $E\subseteq \Sigmacirc (_{\rho}M;A)$. There is a neighborhood
$N$ of $\rho$ in ${\mathcal R}_{E}$ such that for all $\rho '\in N,$
$E\subseteq \Sigmacirc (_{\rho '}M;A)$. Moreover, we can choose $N$
so that there is a uniform constant $\nu >0$ and a $G$-volley 
$\Phi :{\bf F}^{n}\to f {\bf F}^{n}$
inducing $\text{\rm id}_A$ such that for each $e\in E$ and $\rho '\in
N$ there is a selection $\varphi_{e,\rho '}$ from $\Phi$ with ${\rm
gsh}_{e}\varphi _{e,\rho '}\geq \nu$.  \end{thm}

Of course, by Theorem \ref{T:6.4} the same holds when $\Sigmacirc $
is replaced by $\Sigma ^n$.

\begin{proof}
The proof of Theorem \ref{T:6.5} is by induction on $n$. The case $n=0$ was proved in \cite{BGe16}
We will use the following ( Lemma 8.1 of \cite{BGe16}):

\begin{lemma}\label{continuous} For given $c\in F$, the valuation
$v_{e,\rho}(c)$ is (jointly) continuous in $(e,\rho )$.  \hfill$\square$
\end{lemma}

To keep notation simple, we prove the $n=1$
case of the Theorem in detail; the general inductive case proceeds in
the same way, and is left to the reader.

So, we assume that $E\subseteq {^{\circ} \Sigma}^{1}(_{\rho}M)$ and
for all $\rho '$ in a neighborhood $N_{0}(\rho)$ of $\rho$ that $E\subseteq
{^{\circ} \Sigma}^{0}(_{\rho '}M)$. By the previous sections we can
assume more:

\begin{enumerate}[(1)]

\item $E\subseteq \Sigma ^{1}(_{\rho}M)$ with uniform constant lag
$\lambda \geq 0$ --- see Remark 8.3 of \cite{BGe16}.

\item There is a $G$-volley $\Phi :F_{0}\to fF_{0}$ lifting $\text{id}_A$
and a number $\nu >0$ such that for every $e\in E$ and every $\rho '
\in N_{0}$ there is a selection $\varphi _{e,\rho '}$ from $\Phi$ with
$\text{gsh}_{e}(\varphi _{e,\rho '})>\nu$ --- see Theorem \ref{T:6.4}.
\end{enumerate}

For any positive integer $k$, $\text{gsh}_{e}(\varphi _{e,\rho '}^{k})\geq k\nu $.
We choose $k$ so that $k\nu \geq \lambda +||\partial \mid {\bF}^{(1)}|| + \delta$ where $\delta > 0$ is arbitrary.  

When $e\in E$ and $\rho ' \in N_{0}$ the endomorphisms $\varphi _{e,\rho '}^{k}$ are selections from the finite $G$-volley $\Phi^k : {\bF}^{(0)} \to f{\bF}^{(0)}$.

For each $x\in X_1$ define $\Pi(x) := \{g^{-1}\varphi ^{k} _{e,\rho '}(g\partial x)\mid 
g\in G, e\in E, \rho ' \in N_{0}\}$.  This is a finite set of cycles, hence of boundaries.  

We fix $x\in X_{1}$ and $p\in \Pi(x)$ for a moment. For each $e\in E$ we 
choose ${\bar c}(e)\in F_1$ such that $\partial({\bar c}(e)) = p$ and
\begin{equation*}
v_{e,\rho}({\bar c}(e)) > v_{e,\rho }(p)-\lambda 
\end{equation*}
Then there is a neighborhood $N(e,\rho )=N(e)\times N_{e}(\rho )$ such that for all $(e',\rho ')\in N(e,\rho )$
\begin{equation*}
v_{e',\rho '}({\bar c}(e)) > v_{e',\rho '}(p)-\lambda > v_{e',\rho '}(p)-k\nu + ||\partial\mid {\bF}^{(n)}||+\delta.
\end{equation*}
Since $E$ is compact, a finite set of neighborhoods $N(e_{i})$ covers
$E$. Write $N=(\bigcap _{i} N_{e_{i}})\cap N_0$, a neighborhood of $\rho $.

Still fixing $x$ and $p$, for each $e\in E$ we choose $i$ such that
$e\in N(e_{i})$, and we define $c(e)$ to be ${\bar c}(e_{i})$. Then

\begin {equation*}
v_{e,\rho '}  (c(e)) > v_{e,\rho '}  (p)-k\nu + ||\partial\mid {\bF}^{(n)}||+\delta
\end{equation*}

\noindent Recall that $\del c(e)=p$. Define 
$$\Psi (x):=\{c(e)\mid e\in E, x\in X, p\in \Pi (x)\}\in fF_1$$
and extend $\Psi $ to the associated canonical volley $F_{1}\to fF_{1}$.

Now we let $x$ and $p$ vary; we need to write $c(e,x,p)$ in place
of $c(e)$.  For $(e,\rho ')\in E\times N$ the additive homomorphism
$F_{1}\to F_{1}$ defined by $$\psi _{e, \rho '}(gx):=gc(\rho '(g^{-1})e,
x, g^{-1}\varphi ^{k}_{e,\rho '}(g\del x))$$ is a selection from the
volley $\Psi $. Moreover $\del \circ \psi _{e, \rho '}=\varphi _{e,
\rho '}^{k}\circ \del $, so this selection extends a previous chain map
selection from the volley $\Phi ^{k}$. A calculation shows that it has guaranteed shift $\geq \delta$.
\end{proof}

\begin{cor}\label{open2} Let $\rho $ be an isometric action on $M$ as above. If
$\Sigma ^{n}(_{\rho }M;A)=\partial M$ then there is
a neighborhood $N$ of $\rho$ such that $\Sigma ^{n}(_{\rho '}M;A)=\partial M$
for all $\rho '\in N$. 
\end{cor}

Our other openness theorem, the second part of Theorem \ref{Theorem D}, is:

\begin{thm}[$\text{Tits Openness}$] $\Sigmacirc (M;A)$ is open in the Tits metric topology on $\del M$.
\end{thm}

\begin{proof}
The proof of this is exactly the same as the corresponding proof
for ${{^{\circ} \Sigma ^{0}}}(M;A)$ in \cite{BGe16}. We briefly recall it here.
The set $\Sigmacirc (M;A)$ can be described as the union of subsets of the form
$$\Sigma (\varphi ):=\{e\mid \text{gsh}_{e}(\varphi)>0\}$$ where $\varphi
$ runs through all  $G$-finitary endomorphisms of  ${\bf F}^{n}$  which
commute with the augmentation $\epsilon $ and satisfy $\text {gsh}_{e}>0$
for some $e\in \Sigmacirc (M;A)$.  The norm of a $G$-finitary map is
always finite, so the theorem follows from Theorem 3.9 of \cite{BGe16}
which asserts that under these conditions  $\Sigma (\varphi )$ is an
open subset of $\del M$ in the Tits metric topology.  
\end{proof}

\section{Connections with Novikov homology}\label{novikov}

\subsection{The Novikov module}

Assume given: an isometric action of $G$ on $M$, an end point $e\in \partial
M$, and a base point $b\in M$.  As we have seen, this action extends 
to a topological action of $G$ on $\partial M$; let $G_e$ denote the 
subgroup of $G$ which fixes $e$.

From now on we allow any commutative ring $K$ as ground ring\footnote{In
fact all our work up to this point goes through for such a ground ring
$K$.} unless we restrict it explicitly.

The {\it Novikov module} $\widehat {KG}^e$ is the $(KG_e, KG)$-bimodule
defined as follows: As a set, it consists of all (possibly) infinite
sums $\sum_{g\in G, r_{g}\in K}r_{g}g$ such that, for any horoball $HB$ at $e$, all
but finitely many of the points $gb$ for which $r_{g}\in R$ is non-zero
lie in $HB$. This definition is independent of $b$.  The abelian group
structure is termwise addition.  The right action of $G$ is by termwise
right multiplication; note that $\widehat {KG}^e$ is preserved under right
multiplication by $g\in G$ because the effect is merely to change the
base point from $b$ to $gb$.  Left multiplication by $g\in G$ preserves
$\widehat {KG}^e$ if and only if $g\in G_e$. One thinks of $\widehat
{KG}^e$ as a sort of ``completion towards $e$'' of the group algebra
$KG$. It is a generalization of what is often called the ``Novikov ring ''; 
however, in the present generality there
is no obvious multiplication which would make $\widehat {KG}^e$ a ring.

\subsection{Novikov chains}

Starting with a controlled based free resolution ${\mathbf F}\thra A$
(over $M$) we consider the homology of the chain complex $\widehat
{KG}^e\otimes _{KG}{\mathbf F}$ of left $KG_e$-modules.  This is the {\it
Novikov homology of } $A$ {\it with respect to } $\rho $ {\it and } 
$e\in \partial M$.

To give this a more geometric interpretation we describe the chain complex
in a different way.  As before, $X_k$ denotes the given basis for $F_k$
and $Y_k=GX_k$ is the corresponding $K$-basis.  A {\it Novikov $k$-chain}
(with respect to $e$) is a (possibly) infinite $k$-chain of the form
$c=\sum_{y\in Y_k}r_{y}y$ such that

\begin{enumerate}[(i)]

\item for every horoball $HB$ at $e$ all but a finite subset of 
supp$_{Y}(c)$ lies over $HB$,
and
\item there is a finite subset $X_{k}(c)$ of $X_k$ such that all the
members of supp$_{Y}(c)$ are (left) $G$-translates of members
of $X_{k}(c)$; i.e. the $X$-support of $c$ is finite.
\end{enumerate}

When $X_k$ is finite the second condition is redundant.  Typically $X_k$
is finite for the values $k\leq n$ of interest, but the second condition
can be important in the next dimension $n+1$.

We write $C^{e}_k$ for the set of all such chains, with the obvious
left $G_e$-module structure.  Thus we get a chain complex $\mathbf C^e$
and we write $H^e_k$ for the corresponding homology.
The map ${\mathbf C^e}\to \widehat {KG}^e\otimes _{KG}{\mathbf F}$ which
rewrites $\sum_{y\in Y_k}r_{y}y$ as $\sum_{x\in X_k}(\sum_{g\in G}r_{g,
x}g)x$ and takes it to $\sum_{x\in X_k}(\sum_{g\in G}r_{g,x}g)\otimes x$
is an isomorphism of chain complexes.  Thus $H^{e}_{k}$ is isomorphic
to  $\text{Tor}_{k}(\widehat {KG}^{e}, A)$ as an $KG_e$-module, and is
therefore independent of the choice of resolution of $A$.

\subsection{Homological characterization of $\Sigmacirc (M;A)$}

\begin{thm}\label{T:10.1} Assume ${\mathbf F}^{(n)}$ is finitely 
generated over $KG$.  Let $e\in \partial M$.
$$e\in \Sigmacirc (M;A) \text{ if and only if } \text{ \rm 
Tor}_{k}(\widehat {KG}^{e'};A)=0 \text { for all } e'\in \text{ \rm 
cl}Ge \text{ and all } k\leq n.$$

\end{thm}

In proving this theorem we can replace $\text{ \rm Tor}_{k}(\widehat
{KG}^{e'};A)$ by $H^{e'}_{k}$.
Let $\varphi :{\mathbf F}^{(n)}\to {\mathbf F}^{(n)}$ be an $K$-chain
map which pushes ${\mathbf F}^{(n)}$ towards $e$.  A {\it Lipschitz
deformation} for $\varphi $ is a $K$-chain homotopy $\sigma :{\mathbf
F}^{(n)}\to {\mathbf F}$ between the identity map and $\varphi $ such that
there exists a function $\nu :{\mathbb R}\to {\mathbb R}$ satisfying

\begin{equation}\label{extra}
v_\gamma (\sigma (y))\geq v_\gamma (y)-\nu (v_\gamma (y))
\end{equation}

for every $y$. This suggests a new $\Sigma $-invariant, namely:

$${\widetilde \Sigma }^{n}(M;A)=\{e\in \partial M\mid \text { 
there is such a push and Lipschitz deformation }\}$$
By definition $\widetilde{\Sigma}^{-1}=\partial M$.

We note that the resolution $\mathbf F$ is a subcomplex of ${\mathbf
C}^e$.  The chain map $${\bar \varphi}^{e}:=1+\varphi +\varphi 
^{2}+\cdots :({\mathbf
C}^{e})^{(n)}\to ({\mathbf C}^{e})^{(n)}$$ is well defined.  The valuation
$v_{\gamma}$ on $\mathbf F$ extends to ${\mathbf C}^e$ in the obvious way,
hence also do such notions as ``guaranteed shift".

\begin{lemma}
$\Sigmacirc (M;A)\subseteq {\widetilde \Sigma}^{n}(M;A)$
\end{lemma}

\begin{proof}
If $\varphi $ is a $G$-finitary  push
of ${\mathbf F}^{(n)}$ towards $e$ then any $G$-finitary chain homotopy
between the identity map and $\varphi$ is a Lipschitz deformation for
$\varphi$.
\end{proof}

\begin{lemma} If $e\in {\widetilde \Sigma }^{n}(M;A)$ then $e\in
{\widetilde \Sigma }^{n-1}(M;A)$ and $H^{e}_{n}=0$.
\end{lemma}

\begin{proof} If $z$ is an $n$-cycle in $\mathbf C^e$ then a calculation gives
$z=\partial {\bar \varphi}^{e}\sigma (z)$.
\end{proof}

\begin{lemma} If $e\in {\widetilde \Sigma }^{n-1}(M;A)$ and $H^{e}_{n}=0$
then $e\in \Sigma^{n}(M;A)$.
\end{lemma}

\begin{proof} The case $n=0$ is clear, so we assume $n>0$.  Let $z\in
F_{n-1}$ be a cycle.  Define $w:={\bar \varphi}^{e}\sigma (z)\in 
C^{e}_n$ where $\sigma $ comes from the ${\widetilde \Sigma}$ 
hypothesis.  Then $\partial w=z$
and (see the inequality (\ref{extra})):
$$v_{\gamma}(w)\geq v_{\gamma}(z) -\nu (v_{\gamma}(z))$$
Since $\mathbf F$ is acyclic in dimension $n-1$, $z=\partial c$ for
some finite $n$-chain $c$.  The chain $w-c$ is a cycle in $C_{n}^{e}$,
and $H_{n}^{e}=0$, so there is a chain $u\in C^{e}_{n+1}$ with
$\partial u=w-c$. The chain $u$ is a finite $KG$-combination of members of
$X_{n+1}$, so there is a number $\lambda \geq 0$ such that $v(\partial
u)\geq v(u)-\lambda $.  Thus we can write  $u=u_{1}+u_{2}$ where
$u_1$ is finite and $$v_{\gamma}(\partial u_2)\geq v_{\gamma}(z) -\nu
(v_{\gamma}(z))$$
Define $c':=w-\partial u_2$.  Then $\partial c'=\partial w=z$, $c'$ is
a finite chain because $c'=\partial u_{1}+c$, and
$v_{\gamma}(c')\geq v_{\gamma}(z)-\nu (v_{\gamma}(z))$.

\end{proof}

\begin{cor}\label{C:10.1} $e\in \Sigmacirc (M;A)$ if and only if, for
all $e'\in \text {\rm cl}Ge, e'\in {\widetilde \Sigma }^{n-1}(M;A)\text{
and } H^{e'}_{n}=0$.
\end{cor}

\begin{proof} Let $e'\in \text {\rm cl}Ge$. We use the Lemmas: $e\in 
\Sigmacirc (M;A)$ implies $e'\in
\Sigmacirc (M;A)$ (by Theorem \ref{T:5.5}), hence $e'\in 
{\widetilde \Sigma }^{n}(M;A)$.  This implies $e'\in {\widetilde 
\Sigma }^{n-1}(M;A) \text{ and } H^{e'}_{n}=0$; hence $e'\in 
\Sigma^{n}(M;A)$, which implies $e\in
\Sigmacirc (M;A)$ by the Characterization Theorem (Theorem \ref{T:6.1}).
\end{proof}

{\it Proof of Theorem \ref{T:10.1}:} It follows from
Corollary \ref{C:10.1} by induction, using the fact that $\text{ \rm
Tor}_{k}(\widehat {KG}^{e'};A)=H^{e'}_{k}$.   \hfill $\square$

\begin{Rem} If we define $\Sigma _{\text { Tor}}^{n}(M;A)$ to be 
the set of end points $e$ such that $H^{e}_{k}=0$ for all $k\leq n$, 
then the lemmas in this section establish the following containments:
$$\Sigmacirc (M;A)\subseteq {\widetilde \Sigma 
}^{n}(M;A)\subseteq \Sigma _{\text { Tor}}^{n}(M;A)\cap 
{\widetilde \Sigma }^{n-1}(M;A)
\subseteq {\Sigma }^{n}(M;A).$$
\end{Rem}

\subsection{Behavior of the dynamical invariant on exact sequences}
Let $A'\rightarrowtail A \twoheadrightarrow A''$ be a short exact 
sequence of finitely generated $KG$-modules.  For each $e\in \partial 
M$ there is an exact coefficient sequence

{\fontsize{8pt}{10}\selectfont
\begin{equation*}
\begin{CD} \hdots @>>>\text{ \rm Tor}_{k}(\widehat {KG}^{e};A') @>>> 
\text{ \rm Tor}_{k}(\widehat {KG}^{e};A) @>>> \text{ \rm 
Tor}_{k}(\widehat {KG}^{e};A'') @>\partial_*>>\text{ \rm 
Tor}_{k-1}(\widehat {KG}^{e};A') @>>> \hdots
\end{CD}
\end{equation*}}

This, together with Theorem \ref{T:10.1} gives:

\begin{thm} Let $A$ and $A'$ be finitely generated and let $e\in 
{\overset\circ\Sigma}{^{n+1}}(M;A'')$. Then
$e\in {\overset\circ\Sigma}{^{n}}(M;A')$ if and only if
$e\in {\overset\circ\Sigma}{^{n}}(M;A)$.
\hfill$\square$
\end{thm}

Since $Tor$ commutes with direct sums, we have:

\begin{prop}  ${\overset\circ\Sigma}{^{n}}(M;A'\oplus
A'')={\overset\circ\Sigma}{^{n}}(M;A')\cap
{\overset\circ\Sigma}{^{n}}(M;A'')$ \hfill $\square$
\end{prop}

\section{Products}\label{S:10}

This section is about the behavior of the $\Sigma$-invariants with
respect to direct products of groups and tensor products of modules.
In particular, we prove Theorem \ref{productthm}. 
\vskip 5pt
The set-up is as follows: We are given
\begin{itemize}
\item ${\bF} \thra A$ and ${\bF'} \thra A'$, admissible free resolutions of  the $KG$-module $A$ and the $KH$-module $A'$ respectively, which are finitely generated in dimensions $\leq n$, and
\item isometric actions of groups $G$ and $H$ 
on proper $CAT(0)$ spaces $M$ and $M'$ respectively.
\end{itemize}
These define a resolution  ${\bF}\otimes _{K} {\bF'} \thra A\otimes _{K} A'$ of the
$G\times H$ module $A\otimes _{K} A'$ and an isometric action of $G\times H$ on the proper $CAT(0)$ space $M\times M'$. Again, this resolution is
finitely generated in dimensions $\leq n$.

We begin by generalizing a theorem of Meinert \cite{Geh}:

\begin{thm}\label{T:11.2}
$$ \Sigmacirc(M\times M';A\otimes _{K} A')^c \subseteq \bigcup^n_{p=0}
{^{\circ} \Sigma ^{p}(M;A)^c *
{^{\circ} \Sigma ^{n-p}}}(M ';A ')^c$$
\end{thm}

For the proof we need a lemma:

\begin{lemma}\label{L:11.1} Let $e\in {^\circ \Sigma 
}{^{k}}(M;A)$ where
$k\leq n$.  There exists $\nu\geq 0$ such that for any $\mu\geq 0$
there is a finitary chain map $\zeta :{\bF}^{(n)}\to {\bF}^{(n)}$ lifting
$\text {\rm id}_A$ such that $\text {\rm gsh}_{e}(\zeta |{\bF}^{(k)} 
)\geq \mu$, and

\begin{equation*}
v_{\gamma}(\zeta (c))\geq v_{\gamma }(c)-\nu.
\end{equation*}
for all $c\in {\bF}^{(n)}$.
\end{lemma}

\begin{proof} Let $\varphi :{\bF}^{(k)}\to {\bF}^{(k)}$ be a $G$-finitary 
chain map inducing id$_A$ and pushing ${\bF}^{(k)}$  towards $e$.  By 
Lemma \ref{L:5.1} $\varphi $ and id are chain homotopic by a 
$G$-finitary chain homotopy $\sigma : {\mathbf F^{(k)}}\to {\mathbf 
F^{(k+1)}}$. Extend $\varphi $ to $\mathbf F^{(n)}$ as 
follows:  Define $\psi _{k+1}:{\mathbf F^{k+1}}\to {\mathbf F^{k+1}}$ 
by $\psi _{k+1}(c)=c-\sigma \partial c$, and define
$\psi (c)=c $ when $c$ has degree $\geq k+2$.  Then $\psi $ is a 
chain map, and if $\sigma $ is extended by defining it to be the zero 
map on chains of degree $\geq k+1$, then $\sigma $ is a $G$-finitary 
chain homotopy  between $\psi $ and id.

Let $r>0$ be such that $r\cdot \text {gsh}_e(\varphi ) >\mu $. We 
will show that $\zeta :=\psi ^{r+1}$ and $\nu :=||\partial 
||+||\sigma ||$ satisfy the requirements of the lemma.  Certainly, 
$\text {gsh}_{e}(\psi ^{r+1}|{\bF}^{(k)} )\geq \mu$.

Let $\gamma $ be a geodesic ray with $\gamma (\infty)=e$. Consider 
the chain homotopy
$$\tau :=\sigma (\text {id} +\psi +\cdots+\psi ^r)$$
between id and $\psi ^{r+1}$. Then for any $c\in {\bF}^{(n)}$ we have

\begin{equation*}
\begin{aligned}
v_\gamma(\tau(c)) &= v_{\gamma }\sigma (\text {id}+\psi +\cdots 
+\psi^{r-1}+\psi^{r})(c)\\
&\geq \min\{v_{\gamma }\sigma (\psi^p\sigma(c))\mid 0\leq p\leq k\}\\
&\geq \min\{v_{\gamma }(c) + p\cdot \epsilon -||\sigma ||\mid 0\leq 
p\leq k\}\\
&=v_\gamma(c)-||\sigma ||.
\end{aligned}
\end{equation*}

If $c$ has degree $k+1$ then $\psi ^{r+1}(c)=c-\tau \partial c$.  So

\begin{equation*}
\begin{aligned}
v_{\gamma }(\psi ^{r+1}(c)) &\geq \min\{v_{\gamma }(c), v_{\gamma }(\tau
\partial c)\}\\
&\geq \min\{v_\gamma(c),  v_{\gamma }(\partial c)-||\sigma ||\}\\
&\geq \min\{v_\gamma(c),  v_{\gamma }(c)-||\partial ||-||\sigma ||\}\\
&=v_{\gamma }(c)-||\partial ||-||\sigma ||.
\end{aligned}
\end{equation*}

And if $c$ has degree $>k+1$ then $\psi ^{r+1}(c)=c$.
\end{proof}

{\it Proof of Theorem \ref{T:11.2}:} By \cite [\S I.5.15]{BrHa99} 
there is a canonical identification of $\partial (M\times M')$ with the join $\partial M
*\partial M'$. Following \cite{BrHa99} page 266, if $e\in \partial M$ and $e'\in
\partial M'$, the $\theta $-point on the join line from $e$ to $e'$ is
denoted by $\text{ cos}\theta \  e+\text { sin}\theta \  e'$ where $0\leq
\theta\leq \frac{\pi}{2}$. Picking base points $b\in M$ and $b'\in M'$
let $\gamma ,\gamma '$ and $\gamma ''$ be the geodesic rays in $M, M'$
and $M\times M'$ determining $e, e'$, and $\text{ cos}\theta \  e+\text{
sin}\theta \  e'$. Then $\gamma ''(t)=(\gamma (t\text {cos}\theta),
\gamma '(t\text {sin}\theta))$.

Assuming $\text { cos}\theta \  e+\text { sin}\theta \  e'\notin 
\bigcup^n_{p=0}
{^{\circ}\Sigma}{^{p}}(M;A)^c *{^{\circ} \Sigma ^{n-p}}(M ';A 
')^c$, we will
show that  $$\text { cos}\theta \  e+\text{ sin}\theta \  e' 
\in  \Sigmacirc(M\times M';A\otimes _{K} A')$$.

{\it Case 1:} $0<\theta <\frac{\pi }{2}$ and
$e\in^{\circ}\Sigma^{n}(M;A)^c$. Let $p$ be the largest
integer such that $e\in^{\circ}\Sigma {^{p-1}}(M;A)$. Thus
$e\in ^{\circ}\Sigma {^{p}}(M;A)^c$, so $e'\in 
{^{\circ}\Sigma}{^{n-p}}(M ';A')$. Then $$v_{\gamma ''}(c\otimes c')=\text
{cos}\theta \  v_{\gamma }(c)+\text {sin}\theta \ v_{\gamma '}(c').$$
(For this one needs $\text{ supp}(c\otimes c')=\text {supp}(c)\times
\text {supp}(c')$ which is true because $\Z$ has no zero divisors.)

Let $\epsilon >0$ be fixed. Let $\nu $ and $\nu '$ be as in Lemma
\ref{L:11.1}.  Choose $\mu $ so that $\text {cos}\theta \  \mu -\text 
{sin}\theta \  \nu '>\epsilon$, and
choose $\mu '$ so that $\text {cos}\theta \  \mu ' -\text {sin}\theta 
\  \nu '>\epsilon$.  By Lemma \ref{L:11.1}
there are finitary chain maps $\zeta :{\bF}^{(n)}\to {\bF}^{(n)}$ lifting
$\text {\rm id}_A$ and  $\zeta ':{\bF'}^{(n)}\to {\bF'}^{(n)}$ lifting
$\text {\rm id}_{A'}$ such that 
$$v_{\gamma }(\zeta (c))\geq v_{\gamma }(c)+\mu \text { for all }c\in {\bF}^{(p-1)} \text { and}$$ 
$$v_{\gamma }(\zeta (c))\geq v_{\gamma }(c)-\nu \text { for all }c \in {\bF}^{(n)}$$

When $c\otimes c'$ has degree $\leq n$ and $c$ has degree $\leq p-1$ then, by
\cite [\S II.8.24]{BrHa99}, we have
\begin{equation*}
\begin{aligned}
v_{\gamma ''}(\zeta (c)\otimes \zeta '(c')) &=\text {cos}\theta 
\  v_{\gamma }(\zeta (c))+\text {sin}\theta \  v_{\gamma '}
(\zeta '( c'))\\
&\geq \text {cos}\theta [v_{\gamma }(c)+\mu]+\text {sin}\theta 
[v_{\gamma '}( c')-\nu ']\\
&=v_{\gamma ''}(c\otimes c')+\text {cos}\theta \ \mu -\text 
{sin}\theta \ \nu '\\
&>v_{\gamma ''}(c\otimes c')+\epsilon
\end{aligned}
\end{equation*}

When $c\otimes c'$ has degree $\leq n$ and $c$ has degree $\geq p$ then a similar discussion using $\zeta '$ gives
$$v_{\gamma ''}(\zeta (c)\otimes \zeta '(c'))>v_{\gamma ''}(c\otimes 
c')+\epsilon.$$

So $\text {\rm gsh}_{e}(\zeta \otimes \zeta ')\geq \epsilon$, and thus
$\text { cos}\theta \  e+\text { sin}\theta \  e'\in \Sigmacirc(M\times
M';A\otimes _{K} A')$.

{\it Case 2:} $0<\theta <\frac{\pi }{2}$ and $e\in
^{\circ} \Sigma {^{n}}(M;A)$.
If $c\otimes c'$ has degree $\leq n$, the above
argument again gives $$v_{\gamma ''}(\zeta (c)\otimes \zeta '(c'))
 >v_{\gamma ''}(c\otimes c')+\epsilon .$$

{\it Case 3:} If $\theta =0$ then $e\in 
{^{\circ} \Sigma}{^{n}}(M;A)$, and $\zeta '$ can be replaced by $\text {\rm 
id}_{\bF'}$ above.  The case $\theta=\frac{\pi }{2}$ is handled 
similarly.    \hfill $\square$

We turn to the opposite inclusion ``$\supseteq $'', starting with the
observation that it cannot hold generally in a situation where $A\neq
0\neq A'$ while $A\otimes  A'=0$. Therefore, from now on we will assume
that $K$ is a field.

\begin{thm}\label{T:11.3} Let $K$ be a field, $A$ a $KG$-module of type $FP_p$ and $A'$ a $KH$-module of type $FP_q$. If $\Sigma ^{0}(M;A)=\del M$ 
and $\Sigma ^{0}(M ';A')=\del M'$ then 
$$\Sigma ^{p}(M;A)^{c} *\Sigma ^{q}(M';A')^{c}\subseteq \Sigma ^{p+q}(M\times M';A\otimes _{K}A')^{c}$$ 
\end{thm} 

\begin{rems} 

(1) The statement that $\Sigma ^{0}(M;A)=\del M$ 
is equivalent to saying that the $G$-action on $M$ is cocompact and $A$ 
has bounded support; see Theorem 9.1 of \cite{BGe16}. When the $G$-action 
has discrete orbits, this reduces to cocompactness together with $A$ being 
finitely generated over the point stabilizer $G_b$ for some (equivalently, 
any) point $b\in M$. See also Theorem \ref{T:7.8} and Corollary \ref{T:7.9}. 
\vskip 10pt 
(2) In \cite{BGe10} we established the product formula for 
$\Sigma ^{n}(G\times H;K)$, i.e. the case where $M=G_{ab}\otimes {\R}$ and $M'=H_{ab}\otimes {\R}$ are Euclidean and $A=K=A'$, where the $0$-dimensional assumptions discussed in the previous remark are trivially satisfied. The proof of Theorem \ref{T:11.3} given below lifts the key arguments of \cite{BGe10} to the $CAT(0)$ case with modules $A, A'$, but this lifting only works when those assumptions hold.
\end{rems}
\vskip 5pt 

\begin{proof} ({\it of Theorem \ref{T:11.3}}) We use the projective
$K(G\times H)$-resolution $\epsilon \otimes {\epsilon '}:{\bF}\otimes
_{K} {\bF'} \thra A\otimes A'$, noting that if $p\neq 0\neq q$ then
the chain arguments used in the proof of Theorem 5.2 of \cite{BGe10}
away from $H_{0}({\bF}\otimes _{K} {\bF'})=A\otimes A'$ carry over {\it
mutatis mutandis}. Therefore, without loss of generality it only remains
to consider the case $q=0$. And, as we assume $\Sigma ^{0}(M ';A')=\del
M'$, we need only show \begin{equation*} \del M \cap \Sigma ^{p}(M\times
M';A\otimes A')\subseteq \Sigma ^{p} (M;A)\tag{**}
 \end{equation*}

\noindent For this we can ignore the $H$-action, choose a $KG$-embedding
$A\rightarrowtail A\otimes A'$ with a $K$-splitting, and lift it to
a $K$-split $KG$-embedding $s:{\bF}\rightarrowtail {\bF}\otimes _{K}
{\bF'}$ with corresponding  projection
$\pi :{\bF}\otimes_{K} {\bF'}\twoheadrightarrow {\bF}$;
we then have $\pi \circ
s=\text{id}_{\bF}$. The horoballs of $M\times M'$ at $e\in \del M\subseteq
\del (M\times M')$ are of the form $HB_{e}(M\times M')=HB_{e}(M)\times
M'$, and the Busemann function $\beta_{e}:M\times M'\to {\R}$ ignores
the $M'$ contribution. Hence the valuation $v_{e}:{\bF}\otimes_{K}
{\bF'}\to M\times M'\to {\R}$ restricts to the corresponding valuation
${\bF}\to {\R}$, and this implies the inclusion (**).  
\end{proof} 
\vskip 5pt
To complete the proof of Theorem \ref{productthm}, i.e. to prove $$
\Sigmacirc(M\times M';A\otimes _{K} A')^c = \bigcup^n_{p=0} {^{\circ}
\Sigma ^{p}(M;A)^c * {^{\circ} \Sigma ^{n-p}}}(M';A')^c,$$
 it only remains to replace $\Sigma $ by ${^\circ}\Sigma $ in Theorem \ref{T:11.3}.
Let 
$$\text{cos}\theta \ e_{0}+\text {sin}\theta \ e_{0}' \in 
{^{\circ}\Sigma} {^{p}}(M;A)^{c} * {^{\circ}\Sigma}{^{n-p}}(M';A')^{c}$$ 
First
assume $0<\theta <\frac{\pi }{2}$.  Then, by the Characterization
Theorem \ref{T:6.1}, $\text{ cl }Ge \cap \Sigma {^{p}}(M;A)^c\neq
\emptyset$ and $\text{ cl }He' \cap \Sigma {^{n-p}}(M ';A')^c\neq
\emptyset$.  Pick $e_{0}\in \text{cl }Ge \cap \ \Sigma {^{p}}(M;A)^c$
and $e'_{0}\in \text{cl }He' \cap \ \Sigma {^{n-p}}(M';A')^c$. Consider
$\text{cos}\theta \ e_{0}+\text {sin}\theta \ e_{0}'$. By Theorem
\ref{T:11.3} this lies in $\Sigma {^{n}}(M\times M'; A\otimes A')^c$.

We are to show $\text{cos}\theta \ e_{0}+\text{sin}\theta \ e_{0}' \in
\Sigmacirc (M\times M'; A\otimes A')^c$. Suppose not. Then $$\text
{ cl}[(G\times H)(\text{cos}\theta \ e_{0} + \text {sin}\theta \ 
e_{0}')]\subseteq \Sigma {^{n}}(M\times M'; A\otimes A'))$$ Since
$e_{0}\in \text{ cl }Ge$, $e_0$ is the limit elements of the form
$ge$. Similarly, $e_{0}'$ is the limit elements of the form $he'$. So
$\text{cos}\theta \ e_{0}+\text {sin}\theta \ e_{0}'$ lies in $\text {
cl}[(G\times H)(\text{cos}\theta \ e_{0}+\text {sin}\theta \ e_{0}')]\subseteq
\Sigma ^{n}(M\times M'; A\otimes A')$. This is a contradiction.

Obvious alterations of this argument cover the cases $\theta =0$ and $\theta =\frac{\pi}{2}$.

\bibliographystyle{amsalpha}
\bibliography{paper}{}

\def\cprime{$'$}
\providecommand{\bysame}{\leavevmode\hbox to3em{\hrulefill}\thinspace}
\providecommand{\MR}{\relax\ifhmode\unskip\space\fi MR }
\providecommand{\MRhref}[2]{%
  \href{http://www.ams.org/mathscinet-getitem?mr=#1}{#2}
}
\providecommand{\href}[2]{#2}
\begin{thebibliography}{BNS87}

\bibitem[B{\Gr}82]{BGr82}
Robert Bieri and J.~R.~J. {\Gr}oves, \emph{Metabelian groups of type {$({\rm
  FP})\sb{\infty }$} are virtually of type {$({\rm FP})$}}, Proc. London Math.
  Soc. (3) \textbf{45} (1982), no.~2, 365--384. \MR{670042 (83m:20070)}

\bibitem[B{\Ge}03]{BGe03}
Robert Bieri and Ross {\Ge}oghegan, \emph{Connectivity properties of group
  actions on non-positively curved spaces}, Mem. Amer. Math. Soc. \textbf{161}
  (2003), no.~765, xiv+83. \MR{1950396 (2004m:57001)}

\bibitem[BG10]{BGe10}
Robert Bieri and Ross Geoghegan, \emph{Sigma invariants of direct products of
  groups}, Groups Geom. Dyn. \textbf{4} (2010), no.~2, 251--261. \MR{2595091}

\bibitem[BG16]{BGe16}
\bysame, \emph{Limit sets for modules over groups on {$CAT(0)$} spaces: from
  the {E}uclidean to the hyperbolic}, Proc. Lond. Math. Soc. (3) \textbf{112}
  (2016), no.~6, 1059--1102. \MR{3537333}

\bibitem[BH99]{BrHa99}
Martin~R. Bridson and Andr{\'e} Haefliger, \emph{Metric spaces of non-positive
  curvature}, Grundlehren der Mathematischen Wissenschaften [Fundamental
  Principles of Mathematical Sciences], vol. 319, Springer-Verlag, Berlin,
  1999. \MR{1744486 (2000k:53038)}

\bibitem[Bie07]{B07}
Robert Bieri, \emph{Deficiency and the geometric invariants of a group}, J.
  Pure Appl. Algebra \textbf{208} (2007), no.~3, 951--959, With an appendix by
  Pascal Schweitzer. \MR{2283437}

\bibitem[BNS87]{BNS87}
Robert Bieri, Walter~D. Neumann, and Ralph Strebel, \emph{A geometric invariant
  of discrete groups}, Invent. Math. \textbf{90} (1987), no.~3, 451--477.
  \MR{914846 (89b:20108)}

\bibitem[BR88]{BRe88}
Robert Bieri and Burkhardt Renz, \emph{Valuations on free resolutions and
  higher geometric invariants of groups}, Comment. Math. Helv. \textbf{63}
  (1988), no.~3, 464--497. \MR{960770 (90a:20106)}

\bibitem[Bro87]{Bn87}
Kenneth~S. Brown, \emph{Trees, valuations, and the {B}ieri-{N}eumann-{S}trebel
  invariant}, Invent. Math. \textbf{90} (1987), no.~3, 479--504. \MR{914847
  (89e:20060)}

\bibitem[BS80]{BS80}
Robert Bieri and Ralph Strebel, \emph{Valuations and finitely presented
  metabelian groups}, Proc. London Math. Soc. (3) \textbf{41} (1980), no.~3,
  439--464. \MR{591649 (81j:20080)}

\bibitem[Geh98]{Geh}
Ralf Gehrke, \emph{The higher geometric invariants for groups with sufficient
  commutativity}, Comm. Algebra \textbf{26} (1998), no.~4, 1097--1115.
  \MR{1612192}

\bibitem[GO07]{GO07}
Ross Geoghegan and Pedro Ontaneda, \emph{Boundaries of cocompact proper {${\rm
  CAT}(0)$} spaces}, Topology \textbf{46} (2007), no.~2, 129--137. \MR{2313068
  (2008c:57004)}

\end{thebibliography}

\end{document}